\documentclass[12pt, a4paper]{amsart}

\usepackage{palatino}

\usepackage{hyperref}
\usepackage{enumerate}

\usepackage{latexsym}
\usepackage{amssymb}
\usepackage{amsmath}

\makeatletter
\def\moverlay{\mathpalette\mov@rlay}
\def\mov@rlay#1#2{\leavevmode\vtop{%
   \baselineskip\z@skip \lineskiplimit-\maxdimen
   \ialign{\hfil$\m@th#1##$\hfil\cr#2\crcr}}}
\newcommand{\charfusion}[3][\mathord]{
    #1{\ifx#1\mathop\vphantom{#2}\fi
        \mathpalette\mov@rlay{#2\cr#3}
      }
    \ifx#1\mathop\expandafter\displaylimits\fi}
\makeatother

\usepackage{fullpage}

\usepackage{amscd}
\usepackage{mathrsfs}
\usepackage{color}

\newtheorem{theorem}{Theorem}[section]
\newtheorem{lemma}[theorem]{Lemma}
\newtheorem{corollary}[theorem]{Corollary}
\newtheorem{proposition}[theorem]{Proposition}

\theoremstyle{definition}

\newtheorem{definition}[theorem]{Definition}
\newtheorem{example}[theorem]{Example}

\theoremstyle{remark}

\newtheorem{remark}[theorem]{Remark}

\DeclareMathOperator{\Aut}{{\text{Aut}}}

\DeclareMathOperator{\Hom}{{\text{Hom}}}
\DeclareMathOperator{\Obj}{{\text{Obj}}}
\DeclareMathOperator{\cod}{{\text{cod}}}
\DeclareMathOperator{\dom}{{\text{dom}}}

\def\ident{1}
\def\triv{\{\ident\}}
\def\NN{\mathbb N}
\def\ZZ{\mathbb Z}

\def\Si{\Sigma}

\title[Scale-multiplicative semigroups and $P$-graphs]{A graph-theoretic description of scale-multiplicative semigroups of automorphisms}
\author[C.E.~Praeger]{Cheryl E. Praeger}
\address[Cheryl E. Praeger]{
School of Physics, Mathematics and Computing\\
The University of Western Australia\\
35 Stirling Highway\\
Crawley\\
WA 6009 \\ 
Australia}%
\email{Cheryl.Praeger@uwa.edu.au}

\author[J.~Ramagge]{Jacqui Ramagge}
\address[Jacqui Ramagge]{
School of Mathematics and Statistics\\
Sydney University\\
NSW 2006\\
Australia}%
\email{Jacqui.Ramagge@sydney.edu.au}
\thanks{}

\author[G.~Willis]{George Willis}
\address[George Willis]{
School of Mathematical and Physical Sciences\\
The University of Newcastle\\
Callaghan\\
NSW 2308 \\ 
Australia}%
\email{George.Willis@newcastle.edu.au}


\thanks{This research was supported by Australian Research Council grants DP150100060 and DP160102323.}

\begin{document}

\begin{abstract}
It is shown that a flat subgroup, $H$, of the 
totally disconnected, locally compact group $G$ 
decomposes into a finite number of subsemigroups 
on which the scale function is multiplicative. 
The image, $P$, of a multiplicative semigroup in the quotient, 
$H/H(1)$, of $H$ by its uniscalar subgroup 
has a unique minimal generating set 
which determines a natural Cayley graph structure on $P$. 
For each compact, open subgroup $U$ of $G$, 
a graph is defined and it is shown that 
if $P$ is multiplicative over $U$ then this graph 
is a regular, rooted, strongly simple $P$-graph. 
This extends to higher rank the result of R.~M\"oller 
that $U$ is tidy for $x$ if and only if a certain graph is a regular, rooted tree.
\end{abstract}

\maketitle

\setcounter{tocdepth}{1}
\tableofcontents

\section{Introduction}

The connected component of the identity of a locally compact group is 
a normal subgroup and the quotient by this subgroup is totally disconnected and locally compact. 
Hence every locally compact group is 
an extension of a connected locally compact group by 
a totally disconnected, locally compact group. 
By the solution to Hilbert's fifth problem in 1952 
from the combined work in~\cite{Gleason52,MontZip52}, 
connected locally compact groups can be 
approximated from above by Lie groups. 
The quest for the understanding of connected locally compact groups 
could then apply linear algebraic techniques via Lie algebras. 
It was not until 1994~\cite{Wi:structure} that 
a corresponding insight was available for 
totally disconnected, locally compact groups.

In~\cite{Wi:structure}, Willis introduced the notions of 
\emph{scale function}  and \emph{tidy subgroups} 
in the context of a totally disconnected, locally compact group~$G$. 
The scale of an automorphism of~$G$ is 
analogous to an eigenvalue for a linear operator, 
with tidy subgroups being compact open subgroups of~$G$ 
that play the r\^ole of eigenspaces. Groups of automorphisms sharing a common tidy subgroup are investigated in~\cite{Wi:SimulTriang} and, following a suggestion of U.~Baumgartner,  
a subgroup of automorphisms of $G$ is said to be \emph{flat} 
if all its elements share a common tidy subgroup. 
Inner automorphisms allow us to 
transfer all these notions 
from automorphisms to elements and subgroups of~$G$. 

In~\cite{struc(tdlcG-graphs+permutations)}, 
M\"oller gave a graphical criterion for 
a compact open subgroup~$U$ 
to be tidy for an element~$x$ of~$G$. 
He defined a graph with vertices being certain cosets of~$U$, 
and proved that $U$ is tidy for $x$ if and only if the graph is a regular rooted tree.
In this paper we explore the extent to which 
M\"oller's graphical characterisation of tidiness 
can be extended to flat subgroups of~$G$. 

The graph constructed by M\"oller involves the semigroup of positive powers of~$x$ and the branching number of the regular tree is equal to the scale of~$x$. Using~$x^{-1}$ in the construction results in a different graph even though~$U$ is tidy for~$x^{-1}$ if it is tidy for~$x$. The regular branching of M\"oller's tree corresponds to the fact that the scale satisfies $s(x^n) =  s(x)^n$ for all $n\geq0$, and our extension to flat subgroups of~$G$ involves subsemigroups which are \emph{scale-multiplicative}, that is, they satisfy $s(xy) = s(x)s(y)$ for all $x,y$ in the semigroup. These semigroups of~$G$ are 
of independent interest as ingredients in the construction of geometries associated to  
totally disconnected, locally compact groups~\cite{BRW:trees}, and M\"oller's tree and its generalisation described here are expected to feature in these geometries.

Geometric intuition and examples suggest that the appropriate generalisation might have dimension greater than~$1$, such as in a product of trees. We shall see, however, that products of trees are not sufficiently general. Our main result is that if~$U$ is a compact open subgroup of~$G$ 
and $P\subset G$ is a semigroup that is 
multiplicative over~$U$ in a sense 
made precise in Definition~\ref{defn:multiplicative over V}, 
then a graph constructed from certain cosets of~$U$ is 
a regular, rooted, strongly simple $P$-graph.
Unfortunately, the converse statement is false;~\cite[Example~3.5]{Wi:SimulTriang}
provides a counterexample.

Following~\cite{BSV}, for a multiplicative semigroup $P$, a $P$-graph is defined as a category.
A graph in the classical sense is an $\NN$-graph in this more general context; 
paths in a graph have a \emph{length} corresponding to an element of~$\NN$. 
In a $P$-graph, paths have a \emph{degree} labelled by elements of~$P$
that is more like a shape than a length.
The special case of $\NN^k$-graphs, or $k$-graphs, 
have played a significant r\^ole in the study of 
$C^*$-algebras for a number of years~\cite{KumPask}.
In the general theory, a $P$-graph is a higher-rank structure like a cell complex;
we will need only the $1$-skeletons of the $P$-graphs we define.
The original context in which $P$-graphs appeared 
led to the inclusion of the condition that $P$ be quasi-lattice ordered. 
This condition is not central to the notion of a $P$-graph, 
which is fortunate since Example~\ref{ex:not_Nk} indicates that we need to consider $P$-graphs associated to semigroups that are not quasi-lattice ordered. 

The paper is organised as follows. 
In section~\ref{sec:example} we motivate and illustrate the ideas we will develop via a graph-theoretic example.
Section~\ref{sec:basics}  summarises the relevant background material and establishes notation. 
We include basic definitions and results from the general theory of 
totally disconnected, locally compact groups in subsection~\ref{sec:Willis thy}. 
The notion of a semigroup being multiplicative over a compact open subgroup 
in Definition~\ref{defn:multiplicative over V} appears also in~\cite{BRW:trees}; 
the concept is new and work on both projects was taking place in parallel. 
In subsection~\ref{sec:Roggi} we summarise 
the relevant results in~\cite{struc(tdlcG-graphs+permutations)}, 
using notation that lends itself to our generalisation. 
Section~\ref{sec:semigroups in flat groups} contains 
some known results about flat groups and 
new results about subsemigroups of flat groups. 
In particular, Proposition~\ref{prop:gen_set} 
establishes the existence of the unique minimal generating set 
referred to in the abstract and used in the construction of the $P$-graph. 
Section~\ref{sec:P-graphs} deals with $P$-graphs and establishes our main result,
Theorem~\ref{thm:P-graph G}. 
Section~\ref{sec:examples} contains examples. 
In the first example, the graph is a product of regular trees. 
The second example shows that the graph need not be 
a product of trees even when the semigroup is isomorphic to $\NN^k$. 
In the third example the semigroup is not isomorphic to $\NN^k$, 
showing the necessity of the generalisation to $P$-graphs. 
In subsection~\ref{sec:tree_prod} we describe 
conditions that ensure the $P$-graph is a product of trees. 
Finally, section~\ref{sec:questions} contains some remarks and open questions.

\section{Instructive graph-theoretic example}
\label{sec:example}  

For each possible dimension, we present a concrete example of a subdegree-finite transitive permutation group, arising as 
automorphisms of an explicitly defined locally finite graph, which  
admits a flat sub-semigroup of that dimension. 

We start with a `one-dimensional example' which fits the theory introduced by 
R\"oggi M\"oller in \cite[Section 3]{struc(tdlcG-graphs+permutations)}.
It is essentially \cite[Example 1 on page 809]{struc(tdlcG-graphs+permutations)},
but we give more details, which allow us to work with the underlying point set.
We follow the conventions of M\"oller and write group actions on the right, 
so for a group $H$ acting on a point set, and a point $v$ and element $h\in H$,
we denote by $v^h$ the image of $v$ under $h$, and by 
$v^H = \{ v^h | h\in H\}$, the $H$-orbit containing $v$. 
An $i$-arc in a graph is a sequence $(v_0,\ldots,v_i)$ of vertices such that $v_i$ is adjacent to $v_{i+1}$ and $v_i\ne v_{i+2}$ for each $i\geq 0$. An automorphism group $G$ is $i$-arc transitive if $G$ is transitive on the set of $i$-arcs.

\begin{example}\label{ex:tree}
{\rm 
Let~$d$ be a positive integer, and let~$T$ be a $(d+1)$-regular tree with 
automorphism group $G=\Aut(T)$. If $d=1$ then~$T$ is simply a two-way infinite path, 
but this degenerate situation also gives a valid, if trivial, example. Let~$v_0$ be a vertex of~$T$.
Choose $X = (v_i)_{i\in\mathbb{Z}}$, a two-way infinite path in~$T$
through $v_0$, and $x\in H$ such that $v_i^x =v_{i+1}$ for all~$i$. Then each
$\alpha_i := (v_{i-1},v_i)$ is an arc of $T$, that is to say, an ordered pair of adjacent vertices.
Let~$U:= G_{\alpha_0}$, the stabiliser of the arc  
$\alpha_0$ (or equivalently the stabiliser of the two vertices $v_0$ and $v_{-1}$), let 
\[
Y:= \bigcup_{i\geq 0} v_i^U,
\]
and let~$[Y]$ denote the subgraph of~$T$ induced on the vertex-subset~$Y$.
Note that, for each $i\geq0$, $i$-arc transitivity of~$G$ implies that $v_i^U$ consists of the~$d^i$ vertices of~$T$ at distance $i$ from $v_0$ and $i+1$ from $v_{-1}$.
Each edge of~$[Y]$ joins a vertex of~$v_i^U$ to a vertex of~$v_{i+1}^U$,
for some~$i$, and we may therefore orient each edge of~$[Y]$ `away from~$v_0$'. In this way~$[Y]$ 
becomes a directed regular rooted tree such that all edges are directed away from the root~$v_0$, all vertices have the same out-valency~$d$, and the in-valency of each vertex, except~$v_0$, is~$1$.  For each~$i$, the $U$-orbit~$v_i^U$ is the set of vertices of~$[Y]$
which are reachable from~$v_0$ by directed paths of length~$i$.  
Note that the stabiliser in~$U$ of~$v_i=v_0^{x^i}$ is~$U\cap x^{-i}Ux^i$.

The group~$G$ is totally disconnected and locally compact relative to the permutation 
topology on the vertex set $V(T)$,  and $U$ is a compact open subgroup. 
We thus have what is needed to construct the digraph~$\Gamma_+$ of
\cite[Section 3]{struc(tdlcG-graphs+permutations)} 
and to identify it with the sub-digraph~$[Y]$ of~$T$ just defined. 
Recall that the vertices of~$\Gamma_+$ are right $U$-cosets: 
precisely, and denoting~$\nu_i = Ux^i$, the vertex set~$V(\Gamma_+)$ and the (directed) edge set~$E(\Gamma_+)$ are
\[
V(\Gamma_+) := \bigcup_{i\geq 0} \nu_i^U, \quad E(\Gamma_+) := \bigcup_{i\geq 0} (\nu_i, \nu_{i+1})^U. 
\]
Identifying~$U$ with the element~$v_0$, which it stabilises, we identify~$\nu_i^u=Ux^i u$ in $V(\Gamma_+)$ with $v_i^u$, for each $i, u$ as follows.  
The map $v_i^u \rightarrow Ux^iu$ determines a well defined bijection $\varphi: V(\Gamma^+)\rightarrow Y$
since, for all~$i$ and for all $u, u'\in U$, $v_i^u = v_i^{u'}$ if and only if $u'u^{-1}$ fixes~$v_i$, or equivalently 
$u'u^{-1}\in U\cap x^{-i}U x^i$; 
since $u, u'\in U$, this holds, in turn, if and only if $Ux^iu=Ux^iu'$. 
Now each arc of~$\Gamma_+$ is of the form $(\nu_i^u, \nu_{i+1}^u) = (Ux^iu, Ux^{i+1}u)$, for some~$i\geq0$ and some~$u\in U$. 
Each such arc is the image under $\varphi$ of the vertex pair $(v_i^u, v_{i+1}^u)$ from~$Y$. 
This is an arc of~$[Y]$ (since~$x, u$ are automorphisms of~$T$), and moreover all arcs of~$[Y]$ are of this form. 
Thus~$\varphi$ is a digraph isomorphism, and so~$\Gamma_+$ is a directed regular rooted tree, rooted at~$\nu_0$,
with out-valency~$d$ and in-valency~$1$ (except for~$\nu_0$).

It follows from \cite[Theorem 3.4]{struc(tdlcG-graphs+permutations)} or \cite[Section~3]{Wi:structure} that $U$ is tidy for $x$, 
and hence, by \cite[Corollary~3]{Wi:structure}, $U$ is simultaneously tidy for each element of $H:=\langle x\rangle$, which is therefore a flat group.  In passing we note that the subgroups
\[
U_+ := \bigcap_{i=0}^\infty x^i U x^{-i},\quad U_- := \bigcap_{i=0}^\infty x^{-i} U x^{i},
\]
defined as in \cite[Definition 1]{struc(tdlcG-graphs+permutations)}, are the stabilisers of the one-way infinite
paths $X_- = (v_i)_{i\leq 0}$, and $X_+ = (v_i)_{i\geq 0}$ in $T$, respectively,
and give rise to the factorisation $U = U_- U_+$. Since $G=\Aut(T)$, it is not difficult to see that 
$U_+$ has the same orbits as $U$ in $Y$. 
}
\end{example}
  
To construct a group with flat-rank $n$, consider a set of $n$ triples
$(G_i, U_i, x_i)$, where each $G_i$ is a totally disconnected 
locally compact group,  $U_i$ is a compact open subgroup of $G_i$, $x_i\in G_i$ has infinite order, and $U_i$ is tidy for $x_i$. 
Possibly, but not necessarily, $(G_i, U_i, x_i)$ could be as in the above example.
Then the direct product $G = G_1\times \ldots,\times G_n$ is also a totally disconnected 
locally compact group, and $U = U_1\times \ldots,\times U_n$ is  a compact open subgroup of $G$
which is simultaneously tidy for each element of the group 
$H:= \{ x_1^{i_1} x_2^{i_2} \dots x_n^{i_n} \mid \  i_j\in\mathbb{Z}\} \cong \mathbb{Z}^n$, and so the group $H$ is therefore flat. 
This construction underpins 
Example 5.1 given later. Here we give a concrete example based on a Cartesian product of finitely many
trees coming from Example~\ref{ex:tree}.

For  $i=1,\dots,n$, let $\Si_i$ be a graph with vertex set $V(\Si_i)$ and edge set $E(\Si_i)$.
Define the \emph{Cartesian product} 
$\Si :=\Si_1\,\Box\,\Si_2\,\Box\,\dots\,\Box\,\Si_n$ as the graph with vertex set $V(\Si)=V(\Si_1)
\times V(\Si_2)\times\dots\times V(\Si_n)$, such that a vertex-pair $\{\nu,\tau\}$ 
is adjacent if and only if there exists $s$ such that $\nu_i=\tau_i$ for $i\ne s$ 
and $(\nu_s,\tau_s)\in E(\Si_s)$. Our example involves a Cartesian product of trees.

\begin{example}\label{ex:Cartdec} {\rm
For each $i=1,\dots,n$, let $(T_i, G_i, X_i, U_i, x_i, d_i, Y_i, \Gamma_{+,i})$ be as in Example~\ref{ex:tree},
and define the graph Cartesian product $T = T_1\,\Box\, T_2\,\Box\,\dots\,\Box\, T_n$, and group direct products
$G = G_1\times \ldots,\times G_n$ and $U = U_1\times \ldots,\times U_n$. As discussed above, 
$G$ is a totally disconnected locally compact subgroup of $\Aut(T)$, and $U$ is a compact open subgroup
which is tidy for the 
flat group 
$H:= \{ x_1^{i_1} x_2^{i_2} \dots x_n^{i_n} \mid \  i_j\in\mathbb{Z}\} \cong \mathbb{Z}^n$. Consider the scale multiplicative sub-semigroup 
$S:=\{ x_1^{i_1} x_2^{i_2} \dots x_n^{i_n} \mid \ \mbox{each}\ i_j\geq 0\}$ of $H$.  
 
To visualise the substructure of $T$ which is equivalent to $\Gamma_+$ in Example~\ref{ex:tree},
let $X_i=( v^i_j)_{j\in \mathbb{Z}}$ such that $(v^i_j)^{x_i} = v^i_{j+1}$ for all $j$, 
and such that $U_i$ is the stabiliser in
$G_i$ of $v^i_{-1}$ and~$v^i_0$.  Consider $\nu := (v^1_0,\dots,v^n_0)$ as the `base vertex' of 
$T$ (it is fixed by $U$). For $i=1,\dots,n$ let $\omega_i$ be the vertex of $T$ which is equal to 
$\nu$ in each entry except the $i^{th}$ entry, where it is equal to $v^i_{-1}$. Thus each
$\alpha^i_0:=(\omega_i,\nu)$ is an arc of $T$, and $U$ is the stabiliser in $G$ of the $n$ arcs 
$\alpha^1_0, \ldots, \alpha^n_0$. 

These arcs determine a `positive cone', namely the sub-digraph
$[Y]$ of $T$ induced on the vertex-subset $Y$ 
defined as follows: $Y$ is the union, over all $\iota = (i_1,\ldots,i_n)$ such that each 
$i_j\geq 0$, of the U-orbit containing $\nu_\iota :=(v^1_{i_1},\ldots,v^n_{i_n})$.
In other words, $Y$ is the Cartesian product $Y_1\times\ldots\times Y_n$. 
Each edge in $[Y]$ joins a vertex of some $\nu_\iota^U$
to a vertex of some `successor'  $\nu_{\iota'}^U$, where $\iota$ and $\iota'$ agree in 
all entries except one, say the~$s^{th}$, and $\iota'_s = \iota_s+1$. We think of such an 
edge as being in the `$s$-direction', and as being directed away from the `base vertex' $\nu$.  
Thus, for each $s$, the direction on the arcs of $[Y_s]$ is inherited in $[Y]$ by the arcs of 
$[Y]$ in the $s$-direction. This discussion shows that the sub-digraph $[Y]$ is the
Cartesian product  $[Y_1]\,\Box\, [Y_2]\,\Box\,\dots\,\Box\, [Y_n]$, and hence $[Y]$
is isomorphic to the Cartesian product  $\Gamma_+:=\Gamma_{+, 1}\,\Box\, 
\Gamma_{+, 2}\,\Box\,\dots\,\Box\, \Gamma_{+, n}$.  Thus $\Gamma_+$ is a Cartesian product
of regular rooted trees $\Gamma_{+, i}$ with out-valency $d_i$, and hence
$\Gamma_{+}$ has constant out-valency $\sum_{i=1}^n d_i$.

Finally we note that the $U$-orbit $\nu_\iota^U$ has cardinality 
$d(\iota) := \prod_{j=1}^n d_j^{i_j}$ and consists of all vertices of $T$ which can be reached
by directed paths of length $\ell(\iota) := \sum_{j=1}^n i_j$ which begin at $\nu$ and involve,
for each $s$, precisely $i_s$ arcs in the $s$-direction. 
}
\end{example}

\section{Background material and notation}
\label{sec:basics}

Let $G$ be a totally disconnected, locally compact group. 

\subsection{Scale function and related results}
\label{sec:Willis thy}

We present the basic results in the context of an arbitrary automorphism $\alpha$ of $G$.

\begin{definition}
The compact open subgroup $V\leq G$ is 
\emph{minimizing} for the automorphism $\alpha$ of $G$ if
$|\alpha(V) : \alpha(V)\cap V|$ is minimal over all compact, open subgroups of $G$.  
The minimum value of this index is called 
the \emph{scale} of $\alpha$ and denoted $s(\alpha)$. 
\end{definition}

Since $\alpha$ is an automorphism of $G$, if $V$ is minimizing for $\alpha$, then $s(\alpha)$ is also equal to 
$|V:V\cap \alpha^{-1}(V)|$ , the minimum length of a non-trivial $V$-orbit. Such orbits underpin M\"oller's combinatorial description in~\cite{struc(tdlcG-graphs+permutations)}.

\begin{definition}
\label{defn:tidy}
A compact, open subgroup $U\leq G$ is \emph{tidy} for $\alpha$ if 
\begin{description}
\item[\bf TA] $U = U_+U_-$, where $U_{+} = \bigcap_{k\geq0} \alpha^{+ k}(U)$, 
$U_{-} = \bigcap_{k\geq0} \alpha^{- k}(U)$; and
\label{thm:tidy1}
\item[\bf TB] $U_{++} := \bigcup_{k\in\mathbb{Z}} \alpha^k(U_+)$ is closed.
\label{thm:tidy2}
\end{description}
\end{definition}

Theorem~3.1 of~\cite{Wi:further} establishes the equivalence of these two concepts.

\begin{theorem}\cite[Theorem~3.1]{Wi:further}
\label{thm:tidy}
The compact open subgroup $U\leq G$ is 
minimizing for $\alpha$ if and only if it is tidy for $\alpha$.
\end{theorem}

Since the proof of the existence of tidy subgroups 
given in~\cite{Wi:structure} is constructive, 
Theorem~\ref{thm:tidy} enables us to 
construct minimizing subgroups rather than just assert their existence.
Although condition {\bf TB} in Definition~\ref{defn:tidy} is asymmetric, 
the following result restores the symmetry.
For the proof, note that $G$ is a locally compact topological group 
and hence has a left-invariant Haar measure which is unique up to a positive scalar and which will be denoted by $\lambda$, see~\cite{Hewitt&Ross}. Then, for each automorphism,~$\alpha$, of~$G$, $\lambda\circ\alpha$ is a left-invariant measure and is therefore equal to $\Delta(\alpha)\lambda$ for some positive real number $\Delta(\alpha)$ called the \emph{module} of $\alpha$. 

\begin{lemma}
\label{lem:tidy for inverse}
$U$ is tidy for $\alpha$ if and only if it is tidy for $\alpha^{-1}$.
\end{lemma}
\begin{proof}
By definition, $\Delta(\alpha) = \frac{\lambda(\alpha(U))}{\lambda(U)}$. Hence
\begin{eqnarray*}
\Delta(\alpha) & 
= \displaystyle{\frac{\lambda(\alpha(U))\lambda(U \cap \alpha(U))}{\lambda(U)\lambda(U \cap \alpha(U))}}
=\displaystyle{\frac{\lambda(\alpha(U))}{\lambda(U \cap \alpha(U))}
\frac{\lambda(U \cap \alpha(U))}{\lambda(U)}} \\
& = \displaystyle{\frac{|\alpha(U) : \alpha(U)\cap U|}{|U : U \cap \alpha(U)|}
= \frac{|\alpha(U) : \alpha(U)\cap U|}{|\alpha^{-1}(U) : \alpha^{-1}(U)\cap U|}}.
\end{eqnarray*}
Since $\Delta(\alpha)$ is independent of $U$, 
$|\alpha(U) : \alpha(U)\cap U|$ is minimised 
if and only if $|\alpha^{-1}(U) : \alpha^{-1}(U)\cap U|$ is minimised, 
and the result follows.
\end{proof}

As a direct result of the proof of Lemma~\ref{lem:tidy for inverse}, we observe the following.

\begin{corollary}
For every automorphism $\alpha$ of $G$, $\Delta(\alpha)=\frac{s(\alpha)}{s(\alpha^{-1})}$.
\end{corollary}

Lemma~2.6 in~\cite{Wi:SimulTriang} shows 
that if $U$ is tidy for $\alpha$ then $U\cap \alpha^n(U) = \bigcap_{k=0}^n \alpha^k(U)$, for all $n\geq1$, 
and it then follows immediately from the definition of tidiness that~$U$ is tidy for $\alpha^n$ and hence for the semigroup generated by~$\alpha$.
We are interested in situations in which $U$ is tidy for groups of automorphisms that are not singly generated.

\begin{definition}
The subgroup $\mathcal{H}$ of $\Aut(G)$ is \emph{flat} if 
there is a compact, open $U\leq G$ that is tidy for every $\alpha\in\mathcal{H}$. 
\end{definition}

Considering inner automorphisms transfers the notions of tidy subgroups and flatness from automorphisms of~$G$ to elements of~$G$. 
In the sequel $s(x)$ will denote the scale of the inner automorphism $\alpha_x : y\mapsto xyx^{-1}$ and $U$ will be said to be tidy for $x$ if it is tidy for~$\alpha_x$.

\begin{definition}
\label{defn:multiplicative over V}
Suppose $V$ is a compact open subgroup of $G$.
A semigroup $S\subseteq G$, (that is to say, a sub-semigroup of $G$), is 
\emph{multiplicative over $V$} 
if 
$V\subseteq S$ and the map 
$m:S\to \ZZ^+$ given by $m(x)= |xVx^{-1}\colon xVx^{-1}\cap V|$ 
satisfies $m(xy)=m(x)m(y)$ for all~$x,y\in S$.
\end{definition}
It is shown in~\cite[Proposition~2.2]{BRW:trees} that, 
if $S$ is multiplicative over $V$, 
then $m(x)$ is equal to the scale of $x$ for every $x\in S$. 
In particular, $V$ is tidy for~$S$ and the scale is multiplicative on~$S$.
The following appears as~\cite[Definition 2.6]{BRW:trees}.

\begin{definition}
\label{defn:s-multiplicative}
A semigroup $S\subseteq G$ is 
\emph{$s$-multiplicative} 
if $s(xy)=s(x)s(y)$ for all $x,y\in S$.
Equivalently, the restriction of the scale function to ~$S$
is a semigroup homomorphism from~$S$ to~$(\ZZ^+,\times)$.
\end{definition}

We end this subsection with a standard result that will be familiar to people from several areas of mathematics. As detailed in Corollary~\ref{for:cosets and scale}, it links the scale of $x^{-1}\in G$ to the number of right $U$-cosets in the double coset $UxU$ whenever $U$ is tidy for~$x$. We denote the set of right $U$-cosets in $G$ by $U\backslash G$. 

\begin{lemma}
\label{lem:double_coset}
Let $U$ be a compact, open subgroup of $G$ and $x\in G$. 
Then 
\begin{enumerate}[(i)]
\item 
\label{lem:double_coset1}
for $u\in U$, $Uxu=Ux$ if and only if $u\in U\cap x^{-1}Ux$, 
\item 
\label{lem:double_coset3}
the number of right $U$-cosets in $UxU$ is equal to $|U : U\cap x^{-1}Ux|$, and 
\item 
\label{lem:double_coset2}
the index of $U\cap yUy^{-1}$ in $U$ is independent of $y\in UxU$.
\end{enumerate}
In particular there is a finite number of right  $U$-cosets in $UxU$.
\end{lemma}
\begin{proof}
For part~\eqref{lem:double_coset1} we have  
$
Uxu = Ux  
\Leftrightarrow 
x^{-1}Uxu = x^{-1}Ux 
\Leftrightarrow
u\in x^{-1}Ux. 
$
Thus $Uxu = Ux$ and  
$
u\in U  
\Leftrightarrow  
u \in U\cap x^{-1}Ux. 
$

For part~\eqref{lem:double_coset3}, consider the transitive action of $U$ on the set of right $U$-cosets in $UxU$ 
given by $u: Uxu'\rightarrow Uxu'u$.  
By part~\eqref{lem:double_coset1}, the stabiliser of $Ux$ in this action is $U\cap x^{-1}Ux$ and hence 
the number of right $U$-cosets in $UxU$ is 
equal to $|U : U\cap x^{-1}Ux|$.

Part~\eqref{lem:double_coset2} follows immediately from part~\eqref{lem:double_coset3}
since $UyU=UxU$ for each $y\in UxU$. 
\end{proof}

\begin{corollary}
\label{for:cosets and scale}
Let $U$ be a compact, open subgroup of $G$ and $x\in G$. 
If $U$ is tidy for $x$ then the number of right $U$-cosets in $UxU$ is equal to $s(x)$.
\end{corollary}
\begin{proof}
Note that $s(x)= |xUx^{-1}\colon xUx^{-1}\cap U| = |U : U\cap x^{-1}Ux|$, 
and apply Lemma~\ref{lem:double_coset}.
\end{proof}

\subsection{M\"oller's characterization of tidiness for a single element}
\label{sec:Roggi}

In~\cite{struc(tdlcG-graphs+permutations)}, 
M\"oller characterised tidiness of subgroups  for a given element of $G$ 
in terms of combinatorial properties of an associated graph.
His construction characterises when a compact open subgroup $U\subset G$ is tidy for an element $x\in G$ as follows.
Let $\Omega = U\backslash G$ be the space of right cosets of $U$ in $G$. 
Denote by $\nu_0 = U\in \Omega$ the trivial coset 
and let $\nu_i = \nu_0x^i=Ux^i$ for $i\in \NN$. 
Then $\nu_i$ can variously be thought of as 
the right coset corresponding to~$x^i$ or 
the image of $\nu_0$ under right multiplication by $x^i$. 
We construct a graph from the orbits of the $\nu_i$ under right multiplication by~$U$. 
Concrete instances of this construction were given in Examples~\ref{ex:tree} and\ref{ex:Cartdec}.  
 
\begin{definition}
The graph $\Gamma_+$ consists of 
vertex set $V\Gamma_+$ and edge set $E\Gamma_+$ defined by
$$
V(\Gamma_+) = \bigcup_{i\geq0} \nu_iU 
\quad\hbox{ and }\quad
E(\Gamma_+) = \bigcup_{i\geq0} (\nu_i,\nu_{i+1})U,
$$
where $(\nu_i,\nu_{i+1})U = \left\{ (\nu_iu,\nu_{i+1}u) \mid u\in U\right\}$. 
\end{definition}

M\"oller characterised tidiness of $U$ for $x$ 
in terms of the structure of~$\Gamma_+$.

\begin{theorem}\cite[Theorem 3.4]{struc(tdlcG-graphs+permutations)}
$U$ is tidy for $x$ if and only if $\Gamma_+$ is 
a directed regular rooted tree with all edges directed away from $\nu_0$.
\end{theorem}

From~\cite[Corollary 3]{Wi:structure}, 
 if $U$ is tidy for $x$ 
then it is tidy for $x^n$ 
and $s(x^n)=s(x)^n$ for~$n\in \NN$. 
Thus M\"oller's result can be thought of as 
a graphical characterisation 
of the tidiness of $U$ 
for the flat subgroup $\langle x\rangle \leq G$
and the multiplicative semigroup $\langle x\rangle_+\subseteq G$. Note that $\Gamma_+$ is defined in terms of $\langle x\rangle_+$ and that its out-valency is equal to the scale of $x$ when $U$ is tidy. 
Our aim in this paper is 
to generalize this characterisation 
from $\langle x\rangle_+\subseteq G$ 
to an arbitrary multiplicative semigroup of elements $S\subseteq G$. 
To do this, we draw on some results on flat subgroups of~$G$.

\section{Flat groups and their subsemigroups}
\label{sec:semigroups in flat groups}

\subsection{Flat groups}

The following result says that if $H$ is 
a finitely generated, abelian group of automorphisms of~$G$ then 
there is a compact, open subgroup $U$ of $G$ that is 
simultaneously tidy for every element of~$H$. 

\begin{theorem}\cite{Wi:SimulTriang}
Every finitely generated, abelian group of automorphisms is flat.
\end{theorem}

In particular, every finitely generated, abelian subgroup of $G$ is flat.
This theorem has since been strengthened to say that 
every finitely generated nilpotent group of automorphisms is flat, and 
every polycyclic group of automorphisms is virtually flat 
in the sense that it has a finite index subgroup that is flat, 
see~\cite[Theorems~4.9 \& 4.13]{ShWi}.

\begin{theorem}\cite[Theorem 4.15]{Wi:SimulTriang}
\label{N normal}
Let $H$ be a finitely generated, flat subgroup of $G$ and 
$U$ be a compact open subgroup of $G$ that is 
simultaneously tidy for all elements of $H$. 
The set $H(1) := \{x \in H : xUx^{-1} = U\}$ is 
a normal subgroup of $H$ and $H/H(1)$ is an abelian,
torsion-free group, no element of which is infinitely divisible. 
Hence 
\begin{equation}
\label{eq:flat}
H/H(1) \cong \ZZ^n. 
\end{equation}
\end{theorem}

\begin{definition}
\label{defn:uniscalar}
The subgroup $H(1)$ is called the 
\emph{uniscalar} subgroup of the flat group~$H$ 
and the exponent $n$ appearing in (\ref{eq:flat}) is 
called the \emph{flat-rank} of~$H$. 
\end{definition}

In the next result the subgroup $U$ is expressed as 
a product of some of its subgroups. 
The subgroups $U_j$ need not be normal and 
cannot, in general, be permuted. 
That is, it is not true in general that $U_iU_j = U_jU_i$. 
The statement that $U = U_0U_1\cdots U_q$ means that 
$U$ is equal to the product of the subgroups in that order. 
Hence for each $u\in U$ there are $u_j\in U_j$ such that $u = u_0u_1\dots u_q$. 
Although the subgroups cannot be permuted, 
the order in which the subgroups appear need not be unique. 
Proposition~\ref{prop:common_factor} below indicates that 
there is a different order corresponding to 
each semigroup that is multiplicative over~$U$. 

\begin{theorem}\cite[Lemma 6.3 and Theorem 6.8]{Wi:SimulTriang}
\label{thm:flat group decomp}
Let $H$ be a finitely generated, flat subgroup of $G$ and 
suppose that $U$ is a compact open subgroup of $G$ that is tidy for $H$. 
Then there are compact subgroups $U_0$, $U_1$, \dots, $U_q$ of $U$ 
such that
\[
U = U_0U_1 \cdots U_q,
\]
where $xU_0x^{-1} = U_0$ for every $x\in H$ and, 
for each $j\in \{1,\dots, q\}$ and every $x\in H$, 
either $xU_jx^{-1}\geq U_j$ or $xU_jx^{-1}\leq U_j$. Moreover, for each $j\in \{1,\dots, q\}$ we have that $U_j = \bigcap \left\{ xUx^{-1} \mid x\in H \text{ and }xU_jx^{-1} \geq U_j\right\}$ and there is $x\in H$ with $xU_jx^{-1} > U_j$.
\end{theorem}
We sometimes refer to the $U_j$ as the \emph{components} of $U$.
We remark that $q=0$ if and only if $U$ is normalised by $H$, and that $U_0$ is redundant if $q>0$ 
because it is equal to $\bigcap_{j=1}^q U_j$ in this case. Cases of Theorem~\ref{thm:flat group decomp} 
when $q>0$ are illustrated by Examples~\ref{ex:tree} and~\ref{ex:Cartdec}. In Example~\ref{ex:tree}, $q=2$ and 
the subgroups $U_j$ are $U_\pm$ while, in Example~\ref{ex:Cartdec}, $q=2n$ and the subgroups $U_j$ of Theorem~\ref{thm:flat group decomp}  
are the subgroups $(U_i)_\pm$, with the $U_i$ ($i=1,\dots, n$) as in Example~\ref{ex:Cartdec}. 
Note that the subgroups  in these examples all permute with each other, namely $(U_i)_{\pm}$ and $(U_j)_{\pm}$ 
centralise each other if $i\ne j$, and $(U_i)_{+}(U_i)_{-} = (U_i)_{-}(U_i)_{+} = U_i$ for each $i$. 
Nevertheless the subgroups $(U_i)_+$ and $(U_i)_-$ do not normalise each other. 

We need some further observations about 
the decomposition in Theorem~\ref{thm:flat group decomp} above.
Part~(\ref{scale decomposition}) of Lemma~\ref{details of decomp} can be found 
in~\cite[Lemma~6.3, Theorem~6.12, Theorem~6.14]{Wi:SimulTriang}.
It is included here with as much discussion 
as is needed to set up notation for later use.

\begin{lemma}
\label{details of decomp}
Assume the notation as in Theorem~\ref{thm:flat group decomp}. 
\begin{enumerate}
\item\label{dilation closed} For each $j\in\{1,\dots, q\}$, 
 $\bigcup_{x\in H} x U_jx^{-1}$ is 
a closed group that is stable under conjug\-ation by~$H$.
\item\label{scale decomposition} For each $j\in\{1,\dots, q\}$, 
define $s_j=\min\{|xU_jx^{-1}:U_j|\colon x\in H, xU_jx^{-1}>U_j\}$. 
Then  there exist homomorphisms $\rho_j\colon H \to \ZZ$ such that 
\begin{equation}
\label{eq:scale_formula}
 s(x) = \prod\left\{ s_j^{\rho_j(x)\vee 0} \mid j\in\{1,\dots, q\}\right\}
\end{equation}
where $\rho_j(x)\vee 0=\max\{\rho_j(x),0\}$.
Moreover, 
$H(1) = \bigcap_{j\in\{1,\dots,q\}} \ker\rho_j$, where $H(1)$ is as in Theorem~\ref{N normal}. 
\end{enumerate}
\end{lemma}

\begin{proof}
\eqref{dilation closed}
To prove that $\bigcup_{x\in H} x U_jx^{-1}$ is a group, it is sufficient to establish that the set $\left\{ x U_jx^{-1} \mid x\in H\right\}$ is linearly ordered by inclusion. 
Let $x_1,x_2\in H$. 
Since $H$ is a group $x_2^{-1}x_1\in H$ and so, 
by Theorem~\ref{thm:flat group decomp},
either $x_2^{-1}x_1U_j(x_2^{-1}x_1)^{-1}\geq U_j$ or 
$x_2^{-1}x_1U_j(x_2^{-1}x_1)^{-1}$ $\leq U_j$.
Hence either $x_1U_jx_1^{-1}\geq x_2U_jx_2^{-1}$ 
or $x_1U_jx_1^{-1}\leq x_2U_jx_2^{-1}$, as required. 
Stability under conjugation by~$H$ is clear and $\bigcup_{x\in H} x U_jx^{-1}$ is closed by~\cite[Proposition 1]{Wi:structure}.

\eqref{scale decomposition} 
All subgroup indices are positive integers because $x U_jx^{-1}$ is compact and, when $U_j$ is a subgroup, it is open. Hence  
there is, for each $j\in \{1,\dots, q\}$, 
an $x_j\in H$ such that $x_jU_jx_j^{-1}> U_j$ and 
$|x_jU_jx_j^{-1} : U_j|$ is a minimum. 
Call this minimum index $s_j$. 
Then $\left\{ x_j^nU_j x_j^{-n}\right\}_{n\in\mathbb{Z}}$ 
is linearly ordered and 
$|x_j^nU_j x_j^{-n} : x_j^m U_j x_j^{-m}| = s_j^{n-m}$
for $m\leq n\in \mathbb{Z}$. 

For each $x\in H$ and $n\in \mathbb{Z}$,  $x_j^{-n}x U_j (x_j^{-n}x)^{-1}$ is either a subgroup or overgroup of~$U_j$. Hence either $xU_jx^{-1} \leq x_j^nU_j x_j^{-n}$ or $xU_jx^{-1} \geq x_j^nU_j x_j^{-n}$ for every $x\in H$ and $n\in\mathbb{Z}$. If 
$$
x_j^nU_j x_j^{-n} < xU_jx^{-1} < x_j^{n+1}U_j x_j^{-n-1}
$$ 
for some $n$, then $|x_j^{-n}x U_j (x_j^{-n}x)^{-1} : U_j| < s_j$, contradicting the choice of $x_j$.  Hence there is $n\in\mathbb{Z}$ such that $xU_jx^{-1} = x_j^nU_j x_j^{-n}$. Defining 
$$
\rho_j(x)  = 
\begin{cases}
\log_{s_j} |xU_jx^{-1} : U_j| & \text{ if } xU_jx^{-1}\geq U_j  \\
-\rho_j(x^{-1}) & \text{otherwise,}
\end{cases}
$$
we obtain a homomorphism   $\rho_j : H \to \mathbb{Z}$ satisfying 
$$
s_j^{\rho_j(x)} = \begin{cases}
|xU_jx^{-1} : U_j|, & \text{ if } xU_jx^{-1}\geq U_j \\
|U_j : xU_jx^{-1} |^{-1}, & \text{ if }xU_jx^{-1}\leq U_j 
\end{cases}.
$$ 
Note that conjugation by $x\in H$ is an automorphism, $\alpha_x$, of the locally compact group $\bigcup_{x\in H} x U_jx^{-1}$ and that $s_j^{\rho_j(x)} = \Delta(\alpha_x)$ is the module of this automorphism. That $\rho_j$ is a homomorphism, or equivalently, the fact that $x\mapsto s_j^{\rho_j(x)} : H \to (\mathbb{R}^+, \times)$ is a group homomorphism, thus corresponds to the fact that the modular function is a homomorphism.

By Theorem~\cite[Theorem~6.12]{Wi:SimulTriang},
the scale of $x$ is 
\begin{equation}
\tag{\ref{eq:scale_formula}}
 s(x) = \prod\left\{ s_j^{\rho_j(x)\vee 0} \mid j\in\{1,\dots, q\}\right\}.
\end{equation}
The statement about the uniscalar subgroup of $H$ follows 
from~\cite[Proposition~6.4]{Wi:SimulTriang}.
\end{proof}

The core of our construction relies on $s$-multiplicative semigroups. Since the scale cannot be both nontrivial and multiplicative on a group, the following result is the best that can be expected for flat subgroups of~$G$.

\begin{lemma}
\label{lem:not_multiplicative}
The scale function is submultiplicative on a flat group~$H$, that is, 
$$
s(xy)\leq s(x)s(y), \quad \text{ for all } x,y \in H.
$$
The inequality is strict if and only if there exists $j\in\{1,\dots, q\}$ such that 
$\rho_j(x)$ and $\rho_j(y)$ are non-zero and have opposite sign, with $\rho_j$ as in Lemma~\ref{details of decomp}.
\end{lemma}
\begin{proof}
By Equation~(\ref{eq:scale_formula}),
$$
s(xy) = \prod\left\{ s_j^{\rho_j(xy)\vee 0} \mid j\in\{1,\dots, q\}\right\}.
$$
Since $\rho_j$ is a homomorphism, 
$$
\rho_j(xy)\vee 0 \leq (\rho_j(x)\vee 0) + (\rho_j(y)\vee 0) \hbox{ for each }j\in\{1,\dots, q\},
$$ 
and it follows that $s(xy)\leq s(x)s(y)$ as claimed. The inequality is strict only if there 
exists $j$ such that $\rho_j(xy)\vee 0 < (\rho_j(x)\vee 0) + (\rho_j(y)\vee 0)$, and this occurs only if~$\rho_j(x)$ and~$\rho_j(y)$ are both non-zero and have opposite sign.
\end{proof}

\subsection{Subsemigroups of flat groups}

We now focus on the implications for subsemigroups of flat groups.
The first result is a simple but useful observation about 
$s$-multiplicative subsemigroups of $H$ that 
follows directly from Lemma~\ref{lem:not_multiplicative} above.
Recall Definition~\ref{defn:s-multiplicative}. 

\begin{corollary}
\label{uniscalar intersection}
Suppose $P$ is an $s$-multiplicative subsemigroup of a flat group $H$ and $x\in H$ has $s(x)>1$. 
Then $P$ cannot contain both $x$ and $x^{-1}$. 
In particular, $P\cap P^{-1}\subseteq H(1)$.
\end{corollary}
\begin{proof}
Since each $\rho_j$ is a homomorphism, $\rho_j(x^{-1})=-\rho_j(x)$. Since $s(x)>1$, it follows from Equation~(\ref{eq:scale_formula})
that $\rho_j(x)>0$ for some $j$. Hence $\rho_j(x^{-1}) <0$. 
Also it follows from Equation~(\ref{eq:scale_formula}) that $s(1)=1$.
If both $x, x^{-1}\in P$, then since $P$ is $s$-multiplicative, 
$1=s(1)=s(x x^{-1}) = s(x) s(x^{-1})$. This however contradicts Lemma~\ref{lem:not_multiplicative}.
Thus, for any $x$ with $s(x)>1$, $P$ contains at most one of $x$ and $x^{-1}$. 
This implies the last assertion. 
\end{proof}

The following definition allows us to keep track of the components of the tidy subgroup  $U$ on which a semigroup can act as an expansion or as a contraction.

\begin{definition}
\label{JP+ and JP-}
For each subsemigroup $P$ of the flat group $H$ define
\begin{eqnarray*}
J_P^+ &=& \left\{ j\in\{1,\dots, q\} \mid \exists x\in P \hbox{ with } \rho_j(x)>0 \right\}\\
J_P^- &=& \left\{ j\in\{1,\dots, q\} \mid \exists x\in P \hbox{ with } \rho_j(x)<0 \right\}.
\end{eqnarray*}
Thus if $j\in J_P^+$ then there is an $x\in P$ for which $xU_jx^{-1}>U_j$ etc.
\end{definition}

\begin{lemma}
\label{lem:multiplicative_supports}
Let $P$ be an $s$-multiplicative subsemigroup of the finitely generated flat group $H$. Then 
\begin{enumerate}
\item \label{opposite}
$J_P^+\cap J_P^- = \emptyset$ 
\item
\label{indicator element}
there exists an $x\in P$ such that 
$\rho_j(x)>0$ for every $j\in J_P^+$ and $\rho_j(x)<0$ for every $j\in J_P^-$, and
\item \label{scale}
either $P\subseteq H(1)$, or $J_P^+\neq \emptyset$ and 
 $s(x) = \prod_{j\in J_P^+} s_j^{\rho_j(x)}$ for every $x\in P$.
\end{enumerate}
\end{lemma}
\begin{proof}
\eqref{opposite}
Suppose there exists $j\in J_P^+\cap J_P^-$. 
Then there are $x\in P$ with $\rho_j(x)>0$ and $y\in P$ with $\rho_j(y)<0$.
By Lemma~\ref{lem:not_multiplicative}, this means $s(xy)<s(x)s(y)$,
contradicting the fact that the scale function is multiplicative on~$P$.
Hence $J_P^+\cap J_P^- = \emptyset$. 
Consequently, for each $j\in\{1,\ldots,q\}$ either 
every element of $P$ acts on $U_j$ in a noncontracting way, 
or every element acts in a nonexpansive way.
Since each $\rho_j$ is a homomorphism into $(\ZZ,+)$, 
this may be restated as the condition that for all $x,y,z\in P\cup \{\ident_G\}$ 
\begin{equation}
\label{eq:order_property}
\rho_j(xyz) \geq \rho_j(y) \hbox{ if }j\in J_P^+ 
\quad \text{ and }\quad 
\rho_j(xyz) \leq \rho_j(y) \hbox{ if }j\in J_P^-.
\end{equation}

\eqref{indicator element} Since $J_P^+\cap J_P^- = \emptyset$ 
each $j\in\{1,\dots, q\}$ is in at most one of $J_P^+$ or $J_P^-$.
If $j\in J_P^+$ choose $x_j\in P$ such that $\rho_j(x_j)>0$ and
if $j\in J_P^-$ choose $x_j$ such that $\rho_j(x_j)<0$.
Now let $x=\prod_{j\in J_P^+\cup J_P^-} x_j$ 
where the product can be taken in any order. 
Then $x\in P$ because $P$ is a semigroup, 
and~(\ref{eq:order_property}) implies that 
$\rho_j(x) \geq \rho_j(x_j)>0$ for every $j\in J_P^+$ and 
$\rho_j(x) \leq \rho_j(x_j) < 0$ for every $j\in J_P^-$ as required. 

\eqref{scale} 
If $J_P^+\neq \emptyset$ then the assertion follows directly from Equation~(\ref{eq:scale_formula}) and the definition of $J_P^+$.
If $J_P^+= \emptyset$, then $\rho_j(x)\vee 0=0$ for all $j\in\{1,\dots, q\}$ and $x\in P$. In this case it follows that $s(x)=1$ for all $x\in P$, and hence $P\subseteq H(1)$.
\end{proof}

\begin{proposition}
\label{prop:max_mult}
Every $s$-multiplicative subsemigroup, $P$, 
of a finitely generated flat group $H$ is contained in a semigroup 
that is maximal by inclusion for that property. 
For such a maximal $s$-multiplicative subsemigroup $P$:
\begin{enumerate}[(i)]
\item    
$J_P^+\cup J_P^- = \{1,\dots, q\}$ and $J_P^+\cap J_P^- = \emptyset$;
\label{prop:max_mult2}
\item $PP^{-1} = H$;  
\label{prop:max_mult3} and
\item $P\cap P^{-1} = H(1)$.
\label{prop:max_mult1}
\end{enumerate}
\end{proposition}
\begin{proof}
The property of being $s$-multiplicative is preserved under increasing unions. 
So, by Zorn's Lemma, every $s$-multiplicative semigroup 
is contained in a maximal one. 

\eqref{prop:max_mult2}
We know that $J_P^+\cap J_P^- = \emptyset$ 
by Lemma~\ref{lem:multiplicative_supports}.
Assume for a contradiction that $J_P^+ \cup J_P^-\not=\{1,\dots, q\}$. 
Choose $\ell\in \{1,\dots, q\}$ with $\ell\not\in J_P^+ \cup J_P^-$.
Then, by the definition of $J^\pm_P$, $\rho_\ell(x)=0$, and hence $xU_\ell x^{-1} =U_\ell$, for all $x\in P$.
Since $\ell\neq 0$, there is $y\in H\setminus P$ with $\rho_\ell(y)>0$ by  
Theorem~\ref{thm:flat group decomp}.

Let $x\in P$ be such that $\rho_j(x)>0$ for every $j\in J_P^+$ and $\rho_j(x)<0$ for every 
$j\in J_P^-$, whose existence is guaranteed by Lemma~\ref{lem:multiplicative_supports}. 
Since $\rho_j(yx^n) = \rho_j(y)+n\rho_j(x)$ for all $j\in\{1,\dots,q\}$ and all $n\geq0$, there exists $n>0$ such that 
\begin{equation}\label{eq:rhoj} 
\rho_j(yx^n)\geq0\quad \mbox{for every $j\in J_P^+$ and\quad $\rho_j(yx^n)\leq0$\quad for every $j\in J_P^-$.}
\end{equation}
To see that $yx^n\not\in P$, observe that $x^n\in P$ (since $P$ is a semigroup)
 and so $\rho_\ell(x^n)=0$. Hence $\rho_\ell(yx^n)=\rho_\ell(y) >0$, which implies that $yx^n\not\in P$.

 \smallskip\noindent
 \emph{Claim:\quad The  subsemigroup, $\tilde P$, of $H$ generated by $P\cup\{yx^n\}$
 is $s$-multiplicative.}

Note that part (i) follows from this claim: since $yx^n\not\in P$, the semigroup $\tilde P$
properly contains $P$, and the maximality of $P$ then implies that $\tilde P$ is not $s$-multiplicative.
Thus proof of the claim gives a contradiction, thereby proving part (i).
 
 \smallskip
We assert that, in order to prove Claim 1, it is sufficient to show that 
$J_{\tilde P}^+\cap J_{\tilde P}^- = \emptyset$.
For if this is true then equation~(\ref{eq:scale_formula}) implies that, for $z\in \tilde P$, 
$$
s(z) = \prod_{j\in J_{\tilde P}^+} s_j^{\rho_j(z)}
$$
from which it follows that the scale is multiplicative on $\tilde P$.
We therefore proceed to show that $J_{\tilde P}^+\cap J_{\tilde P}^- = \emptyset$.
It follows from \eqref{eq:rhoj}  that $J_P^+\subseteq J_{\tilde P}^+$ and $J_P^-\subseteq J_{\tilde P}^-$.
First we show that 
$$
J_{\tilde P}^+ = 
J_P^+ \cup \left\{ j\in\{1,\dots, q\}\setminus J_P^- \mid \rho_j(y)>0\right\}.
$$
If $j\not\in J_P^+\cup J_P^-$ and $\rho_j(y)>0$, then $\rho_j(x)=0$ (by the definition of $J^\pm_P$) 
and as $\rho_j$ is a homomorphism, $\rho_j(yx^n) = \rho_j(y)>0$. Thus $j\in J_{\tilde P}^+$ (by its definition),
so $J_{\tilde P}^+$ contains the right hand side.
Conversely, let $j\in J_{\tilde P}^+\setminus J_P^+$. Then, by the definition of $J_{\tilde P}^+$,
there exists $z\in\tilde P$ such that $\rho_j(z)>0$, and we note that 
$z\in \tilde P \setminus P$ since $j\not\in J^+_P$. Also by the definition of 
$J^+_P$, $\rho_j(x')\leq 0$ for all $x'\in P$. 
Now $z$ is a word in $P\cup\{yx^n\}$,
and the number $s$ of occurrences of $yx^n$ in this word is at least 1 since $z\not\in P$.
Since $\rho_j$ is a homomorphism and since  $\rho_j(x')\leq 0$ for all $x'\in P$,
it follows that $\rho_j(z) \leq s\rho_j(yx^n) \leq s\rho_j(y)$. Therefore, since both $s$ and 
$\rho_j(z)$ are positive, it follows that $\rho_j(y)>0$. Suppose finally, for a contradiction, that 
$j\in J_P^-$. Then by \eqref{eq:rhoj}, $\rho_j(yx^n)\leq 0$. Since also  
$\rho_j(x')\leq 0$ for all  $x'\in P$, it follows that $\rho_j(z')\leq 0$ for all $z'\in\tilde P$, contradicting the fact that
$\rho_j(z)>0$. 
Thus the equality for $J_{\tilde P}^+$ is proved.  A similar proof shows that 
$$
J_{\tilde P}^- = 
J_P^- \cup \left\{ j\in\{1,\dots, q\}\setminus J_P^+ \mid \rho_j(y)<0\right\}.
$$
Consider $J_{\tilde P}^+\cap J_{\tilde P}^-$. 
Since $\rho_j(y)$ can't be both strictly positive and strictly negative, 
$\left\{ j\in\{1,\dots, q\}\setminus J_P^- \mid \rho_j(y)>0\right\} \cap
\left\{ j\in\{1,\dots, q\}\setminus J_P^+ \mid \rho_j(y)<0\right\} =\emptyset$.
Since we have that  $J_P^+\cap J_P^- = \emptyset$ and 
the other two intersections that need to be considered are empty by definition, it follows that $J_{\tilde P}^+\cap J_{\tilde P}^- = \emptyset$ as claimed.

\eqref{prop:max_mult3} 
Since $H$ is a group and $P$ is a subsemigroup of $H$, $PP^{-1}\subseteq H$. 
We must show $H\subseteq PP^{-1}$.  
If $P$ is maximal and $y\in H$ then, by the argument in the proof of~\eqref{prop:max_mult2} 
(especially \eqref{eq:rhoj} and the Claim), 
there is $x\in P$ and $n\geq0$ such that $yx^n \in P$. 
Since $P$ is a semigroup, $x^n\in P$ and so $x^{-n}\in P^{-1}$. 
Hence $y=(yx^n)x^{-n}\in PP^{-1}$ and so $H=PP^{-1}$.

\eqref{prop:max_mult1}
We know that $P\cap P^{-1} \subseteq H(1)$ by Corollary~\ref{uniscalar intersection}.
To see that $P\cap P^{-1} = H(1)$,  it suffices to show that $H(1) \subseteq P\cap P^{-1}$.
We begin by proving that $H(1)\subseteq P$.
Note that the product $PH(1)$ is also a subsemigroup of $H$ 
because $H(1)$ is a normal subgroup of $H$. 
By~Lemma~\ref{lem:not_multiplicative}, $s$ is multiplicative on $PH(1)$.
If $H(1)\not\subseteq P$ then $PH(1)>P$, contradicting the maximality of $P$.
Hence $H(1)\subseteq P$. 
Since $H(1)$ is a group, it follows that $H(1)\subseteq P\cap P^{-1}$  as required.
\end{proof}

For future reference, the argument in the proof of Proposition~\ref{prop:max_mult} 
yields the following.

\begin{corollary}\label{cor:max_mult}
Let $P$ be an $s$-multiplicative subsemigroup  
of a finitely generated flat group $H$, and let $z\in H$ be such that $\rho_j(z)\geq0$ for all 
$j\in J_P^+$ and  $\rho_j(z)\leq0$ for all 
$j\in J_P^-$. Then the subsemigroup generated by $P\cup\{z\}$ is also $s$-multiplicative.
\end{corollary}

The following result shows that, if $U$ is tidy for $H$, then every maximal $s$-multiplicat\-ive 
subsemigroup of $H$ induces a decomposition of $U$.

\begin{proposition}
\label{prop:common_factor}
Suppose that $H\leq G$ is finitely generated and flat, and that the compact open subgroup $U$ is tidy for $H$. Let $P\subset H$ be a subsemigroup of $H$ that is $s$-multiplicative over $U$ and is maximal for that property. Then $U$ is the product of subgroups 
$$
U = U_+U_-,
$$
where  $xU_+x^{-1} \geq U_+$ and $xU_-x^{-1} \leq U_-$ for every $x\in P$. 
\end{proposition}
\begin{proof}
By Lemma~\ref{lem:multiplicative_supports}, there is $y\in P$ such that $\rho_j(y)>0$ for every $j\in J_P^+$ and $\rho_j(y)<0$ for every $j\in J_P^-$. Put 
$$
U_+ = \bigcap_{n\geq0} y^nUy^{-n} 
\quad \text{ and } \quad  
U_- = \bigcap_{n\leq 0} y^nUy^{-n}.
$$
Then, noting that $yU_jy^{-1} = U_j$ if $j\not\in J_P^+\cup J_P^-$, it follows from Theorem~\ref{thm:flat group decomp} that, 
$$
U_+ = \prod\left\{ U_j \mid j\not\in J_P^-\right\} 
\quad \text{ and } \quad  
U_- = \prod\left\{ U_j \mid j\not\in J_P^+\right\}.
$$
Hence, by the definitions of $\rho_j$ and $J_P^-$, if $j\not\in J_P^-$ then $xU_jx^{-1} \geq U_j$ for all $x\in P$, so that 
$$
xU_+x^{-1} = \prod\left\{ xU_jx^{-1} \mid j\not \in J_P^-\right\} \geq U_+,
$$
and, similarly, $xU_-x^{-1}\leq U_-$.
\end{proof}

Note that, in general, not every subset of $\{1,\dots, q\}$ can occur as $J_P^+$ for some scale-multiplicative semigroup~$P$. 
Consider $G = \Aut(T)$ and $x\in G$ a translation on an infinite path, as in Example~\ref{ex:tree}. The group $H = \langle 
x\rangle$ is flat and any subgroup $U$ tidy for $H$ factors as $U = U_+U_-$. No subsemigroup of $H$ can be expanding on both factors.

Since each $\rho_j$ is a homomorphism, 
we can define a homomorphism $Q : H \to \mathbb{Z}^q$ (with $q$ as in Theorem~\ref{thm:flat group decomp}) by 
\begin{equation}
\label{eq:defineQ}
Q(x) = (\rho_j(x))_{j=1}^q.
\end{equation}
In light of \eqref{eq:scale_formula}, the kernel of $Q$ is $\ker Q = H(1)$,
the uniscalar subgroup of $H$, and so $Q$ induces 
an embedding of $H/H(1)$ into $\mathbb{Z}^q$, and by Theorem~\ref{N normal}, $H/H(1) \cong \mathbb{Z}^k$,
for some $k\leq q$.

Before proceeding we prove the following technical result about subsets of~$\NN^q$.

\begin{lemma}
\label{lem:order_on_Nq}
Denote by $\leq$ the component-wise partial order on $\NN^q$ 
whereby $(x_1,\ldots,x_q)\leq (y_1,\ldots, y_q)$ 
if and only if $x_i\leq y_i$ for each $i\in\{1,\ldots,q\}$.
Let $\mathcal{V}$ be an infinite subset of $\mathbb{N}^q$. 
Then there are $x,y\in \mathcal{V}$ with $x\leq y$.
\end{lemma}
\begin{proof}
The proof is by induction on $q$ and the base case, when $q=1$, is clear because~$\mathbb{N}$ is well-ordered. Assume that the claim has been established for some value of $q$ and let $\mathcal{V}$ be an infinite subset of $\mathbb{N}^{q+1}$. Choose $x\in \mathcal{V}$. Then either there is $y\in \mathcal{V}$ with $x\leq y$, in which case we are done, or $x\not\leq y$ for every $y\in \mathcal{V}$. In the latter case, there is $j\in \{1,\dots, q+1\}$ such that $y_j\leq x_j$ for infinitely many $y\in \mathcal{V}$. Since there are only finitely many possible values less than $x_j$, it follows that there is $n< x_j$ such that $y_j = n$ for infinitely many $y\in \mathcal{V}$. Then, by the inductive hypothesis, there are $x',y'\in \mathcal{V}$ with $x'_j = y'_j = n$ and $x'_l\leq y'_l$ for every $l\ne j$. This $x'$ and $y'$ establish the claim for $q+1$ and the result follows by induction. 
\end{proof}
	 
We end the section with a result about generating sets for $Q(P)$, where~$Q$ is as defined in Equation~\eqref{eq:defineQ}.

\begin{proposition}
\label{prop:gen_set}
Let $H$ be a finitely generated flat group and
let $P$ be a maximal multiplicative subsemigroup of $H$.
Define subsemigroups
\begin{eqnarray*}
P_+ &=& \left\{x\in P \mid \rho_j(x)\geq0\hbox{ for all }j\in \{1,\dots, q\}\right\}\\
\hbox{ and }\ P_- &=& \left\{x\in P \mid \rho_j(x)\leq0\hbox{ for all }\ j\in \{1,\dots, q\}\right\}.
\end{eqnarray*}
Then there is a subsemigroup $P_0$ such that $P = P_+P_0P_-$ and each of $P_+\cap P_0$, $P_+\cap P_-$ and $P_0\cap P_-$ 
is equal to $H(1)$. Furthermore, $Q(P)$ is a finitely generated subsemigroup of $\mathbb{Z}^q$ that has a unique minimal generating set
$$
\Sigma = \Sigma_+ \cup \Sigma_0 \cup \Sigma_-,
$$ 
where $\Sigma_+$,  $\Sigma_0$ and $\Sigma_-$ are minimal generating sets of $Q(P_+)$, $Q(P_0)$ and $Q(P_-)$ respectively. 
\end{proposition}
\begin{proof}
Define a partial order on $Q(P)$ by 
$$
Q(x) \leq Q(y) \Leftrightarrow \rho_j(x) \leq \rho_j(y)\hbox{ if }j\in J_P^+\hbox{ and }\rho_j(x) \geq \rho_j(y)\hbox{ if }j\in J_P^-.
$$
(This is the same as the component-wise partial order on $\mathbb{N}^q$ 
if $\rho_j(x)$ is replaced by $-\rho_j(x)$ for $j\in J_P^-$.) 
Then for $x,y\in P$ we have
$$
Q(xy) = Q(x) + Q(y) \geq  Q(x)\vee Q(y).
$$
It follows that, if $Q(x)$ is minimal in $(Q(P),\leq)$, 
then $Q(x)$ must belong to any generating set for $Q(P)$. 

Conversely, suppose that $x, y\in P$ and $Q(x) \leq Q(y)$. Then $z = x^{-1}y\in H$, $\rho_j(x^{-1}y)\geq0$ 
if $j\in J_P^+$ and $\rho_j(x^{-1}y)\leq0$ if $j\in J_P^-$. 
Thus, by Corollary~\ref{cor:max_mult} and the maximality of $P$,  $z\in P$. Also $Q(z)\leq Q(y)$. 
If $Q(x)$ and $Q(z)$ are both minimal, 
then $Q(y) = Q(x)+Q(z)$ is the sum of minimal elements of $Q(P)$. 
Should they not both be minimal, 
they can be decomposed as the sum of still smaller elements. 
This must terminate at some point because $Q(P)\subset \mathbb{N}^q$ and, 
when it does, $Q(y)$ is expressed as the sum of finitely many minimal elements of $Q(P)$. 
Therefore the set of minimal elements in $(Q(P),\leq)$ is 
the unique smallest generating set for $Q(P)$, which we denote by $\Sigma$. 

That $\Sigma$ is finite follows from Lemma~\ref{lem:order_on_Nq}.

Similarly, $Q(P_+)$ is generated by the minimal elements in $(Q(P_+),\leq)$. 
Note that $Q(x)$ is minimal in $(Q(P_+),\leq)$ if and only if it is minimal in $(Q(P),\leq)$ 
and belongs to $Q(P_+)$. A similar statement holds for $Q(P_-)$. Hence 
$$
\Sigma = \Sigma_+ \cup \Sigma_0 \cup \Sigma_-,
$$ 
where $\Sigma_+$ and $\Sigma_-$ are the unique smallest generating sets for $Q(P_+)$ and $Q(P_-)$ respectively, and $\Sigma_0 = \Sigma \setminus (\Sigma_+\cup \Sigma_-)$ equals 
$$
 \left\{ Q(x) \in \Sigma \mid \exists j\in J_P^+,\ l\in J_P^-\hbox{ with }\rho_j(x)>0\hbox{ and }\rho_l(x)<0\right\}.
$$
Let $P_0$ be the subsemigroup of $P$ generated by $\left\{ x\in P \mid Q(x)\in \Sigma_0\right\} \cup H(1)$. Then
$$
P = P_+P_0P_-
$$
because $H$ is abelian modulo $H(1)$. 
Finally, since $H(1)$ is contained in $P_+$ and $P_-$, 
we have $P_+\cap P_-$, $P_+\cap P_0$ and $P_0\cap P_-$ all equal to $H(1)$
since they cannot be greater by Lemma~\ref{lem:not_multiplicative}. 
\end{proof}

\section{$P$-graphs}
\label{sec:P-graphs}

Scale methods tell us little about uniscalar elements. 
If $H$ is flat and $H(1)$ is the group of uniscalar elements of $H$, 
then, by Theorem~\ref{N normal}, $H/H(1)\cong \ZZ^k$, for some $k$, and, 
by Proposition~\ref{prop:max_mult}\eqref{prop:max_mult1}, $P\cap P^{-1} = H(1)$
whenever $P$ is a maximal $s$-multiplicative subsemigroup.
Later we shall assume that the flat group $H\cong \ZZ^k$, and 
that~$P$ is a subsemigroup of~$H$ such that $P\cap P^{-1}$ is trivial.
Under these conditions, the generating set~$\Sigma$ 
constructed in Proposition~\ref{prop:gen_set} is a generating set for~$P$
(since the kernel of~$Q$ is~$H(1)$).

\subsection{The general theory of $P$-graphs}

We begin with the definition of an arbitrary $P$-graph.
Our definition differs slightly from that given in~\cite{BSV}. 
In particular~\cite[Definition 2.1]{BSV} requires that $(H,P)$ be a 
\emph{quasi-lattice ordered group}, that is, 
under the partial order $x\leq y\Leftrightarrow x^{-1}y\in P$ on $H$, 
every pair of elements $x,y\in H$ with a common upper bound in $H$ 
has a least common upper bound. 
Not all semigroups arising here have this least upper bound property, 
see Example~\ref{ex:not_Nk}. 
Luckily, the quasi-lattice ordered condition is 
not an inherent requirement for the definition.
This reduces the r\^ole of $H$ in~\cite[Definition 2.1]{BSV}
to being an ambient group for $P$. 
Initially the only significance of this condition is that  
$P$ must therefore be cancellative.
In our examples, $H\cong \ZZ^k$ so that 
$P$ must necessarily be commutative.
However, we will not 
assume commutativity until it is necessary to do so.

As in~\cite[Definition 2.1]{BSV}, we include the maps $\dom$ and $\cod$ 
in the definition of a $P$-graph.
This is not strictly necessary, but will be efficient.

\begin{definition}
\label{defn:P-graph}
Let $P$ be a semigroup that embeds in a group and satisfies $P\cap P^{-1} = \triv$.
Then a \emph{$P$-graph}, $(\mathscr{L}, d)$ consists of a countable category 
$$
\mathscr{L} = (\Obj(\mathscr{L}), \Hom(\mathscr{L}),\cod,\dom)
$$ 
together with a map $d: \Hom(\mathscr{L}) \to P$, called the \emph{degree map}, 
which satisfies the \emph{factorization property}: 
for every $\lambda\in \Hom(\mathscr{L})$ and $x,y\in P$ with $d(\lambda) = xy$, 
there are unique elements $\lambda_1,\lambda_2\in \Hom(\mathscr{L})$ such that 
$\lambda = \lambda_1\lambda_2$ and $d(\lambda_1) = x,\ d(\lambda_2) = y$.

The identifying maps $\cod, \dom\colon \Hom(\mathscr{L}) \to \Obj(\mathscr{L})$ 
are such that, for each $\lambda\in \Hom(\mathscr{L})$, 
$\lambda\colon \dom(\lambda)\mapsto\cod(\lambda)$.
\end{definition}

In the definition we carefully refer to morphisms $\lambda$ as members of $\Hom(\mathscr{L})$. Occasionally we may simply say that $\lambda\in\mathscr{L}$.  Also we sometimes refer to the $P$-graph  $(\mathscr{L}, d)$ simply as  $\mathscr{L}$ if the degree map $d$ is clear from the context.

\begin{example}
When 
$P = \mathbb{N}$, 
$\mathbb{N}$-graphs correspond to directed graphs. 

Let $X=(V,E)$ be a directed graph with vertex set $V$ and edge set $E$. 
The corresponding category $\mathscr{X}$ has $\Obj(\mathscr{X}) = V(X)$ and 
$\Hom(\mathscr{X})$ equal to the set of 
all finite directed paths (concatenations of edges) in $X$. 
Define the degree $d(\lambda)$ of a path $\lambda$ 
to be the length (number of edges) of the path. 
The factorization property is simply the statement that, 
if $\lambda$ is a path of length $n$ and $n = n_1 + n_2$, 
then $\lambda$ is the concatenation of two subpaths 
$\lambda_1$ of length $n_1$ and $\lambda_2$ of length $n_2$ in $\Hom(\mathscr{X})$.
Because of the applications to operator algebras, traditionally 
$\lambda_1$ is the final $n_1$ edges in $\lambda$, 
and $\lambda_2$, the first $n_2$ edges in $\lambda$.

Conversely, if $\mathscr{X}$ is an $\mathbb{N}$-graph, 
then a directed graph $X$ may be defined by taking 
$V(X)= \Obj(\mathscr{X})$ and 
$E(X)$ to be the morphisms of degree~$1$ in $\mathscr{X}$. 
\end{example}

Every $k$-graph in the sense defined in \cite{KumPask} is an $\mathbb{N}^k$-graph in the sense just defined. Particular examples of $\mathbb{N}^k$-graphs may be formed as products of $\mathbb{N}$-graphs through the product construction described below in Section~\ref{sec:ProductsofPgraphs}. Section~\ref{sec:examples} contains examples of $P$-graphs where $P$ is a subsemigroup of $\mathbb{Z}^k$ not isomorphic to $\mathbb{N}^k$. 

In using $P$-graphs to help characterise $s$-multiplicative semigroups, 
we will need some conditions analogous to those used by M\"oller for the single-element case.

\begin{definition}
\label{defn:regular} Let $(\mathscr{L},d)$ be a  $P$-graph, with $P$ being finitely generated with generating set $\Sigma = \{x_1,\dots, x_n\}$. 
\begin{enumerate}
\item 
The \emph{$1$-skeleton} of  $(\mathscr{L},d)$ with respect to $\Sigma$ 
is the directed graph $L$ with \\
$V(L) = \hbox{Obj}(\mathscr{L})$ and 
$E(L) = \left\{ \lambda\in \hbox{Hom}(\mathscr{L}) \mid d(\lambda)\in \Sigma\right\}$.

\item 
$(\mathscr{L},d)$ is \emph{acyclic} if  
$\lambda\in \hbox{Hom}(\mathscr{L})$ with  
$\hbox{dom}(\lambda) = \hbox{cod}(\lambda)$ 
implies $d(\lambda)=1$, where $1$ denotes the identity in~$P$.

\item For each $\alpha\in \hbox{Obj}(\mathscr{L})$ the \emph{descendant $P$-graph} $\mathscr{L}^\alpha$ has 
\begin{eqnarray*}
\hbox{Obj}(\mathscr{L}^\alpha) &=&  \left\{ \beta\in \hbox{Obj}(\mathscr{L})\mid \exists \lambda\in \hbox{Hom}(\mathscr{L}) \hbox{ with }\lambda : \alpha\mapsto\beta\right\}\\
\hbox{ and } \hbox{Hom}(\mathscr{L}^\alpha) &=& \left\{ \lambda\in \hbox{Hom}(\mathscr{L}) \mid \hbox{dom}(\lambda), \hbox{cod}(\lambda)\in \hbox{Obj}(\mathscr{L}^\alpha)\right\},
\end{eqnarray*}
and the degree map is the restriction of $d$ to $\Hom(\mathscr{L}^\alpha)$.

\item An object $\alpha$ with $\mathscr{L}^\alpha = \mathscr{L}$ is called a \emph{generator} for $\mathscr{L}$; equivalently, $\alpha$ is a generator 
if, for every $\beta\in \hbox{Obj}(\mathscr{L})$, there is a morphism $\lambda : \alpha\mapsto \beta$. 
If the generator is unique it is called the \emph{root} of $\mathscr{L}$, 
and $\mathscr{L}$ is said to be \emph{rooted}.

\item $(\mathscr{L},d)$ is \emph{strongly simple} if there is at most one morphism $\lambda : \alpha\mapsto \beta$ 
for any $\alpha,\beta\in \hbox{Obj}(\mathscr{L})$. 

\item $(\mathscr{L},d)$ is \emph{regular} if for every $\alpha,\beta\in \hbox{Obj}(\mathscr{L})$ 
there is an isomorphism $\phi : \mathscr{L}^\alpha\to \mathscr{L}^\beta$. 

\item The cardinality $n=|\Sigma|$ of the generating set is called the \emph{rank} of $(\mathscr{L},d)$. 
\end{enumerate}
\end{definition}

\begin{remark}
If $P$ has only one generator then every $P$-graph has rank $1$ and is a directed graph.
If $P$ has more than one generator then 
a $P$-graph is an inherently higher-rank object; 
the degree map and the factorisation property determine the higher-rank cells.
The case of a singly-generated semigroup 
considered by M\"oller in~\cite{struc(tdlcG-graphs+permutations)}
corresponds to an $\NN$-graph, hence is equal to its $1$-skeleton.
\end{remark}

\begin{remark}
Strong simplicity of a $P$-graph is 
equivalent to simplicity of the graph obtained by 
defining \emph{every} morphism to be an edge. 
Strong simplicity of a $P$-graph implies it is acyclic. 

A regular (downward directed) rooted tree is 
a rooted, regular and strongly simple $\mathbb{N}$-graph. 
If every edge in the tree is doubled, 
then the $\mathbb{N}$-graph (directed graph) obtained is acyclic, rooted and regular but not strongly simple. Hence strong simplicity of a $P$-graph is not implied by the factorization property and acyclicity. 

Consider the directed lattice $\mathbb{N}^2$, that is, 
the Cayley graph of the semigroup $\mathbb{N}^2$ for the generators $(1,0)$ and $(0,1)$. 
This can be considered as both an $\mathbb{N}$-graph, 
in which case the degree map takes values in $\mathbb{N}$, 
and as an $\mathbb{N}^2$-graph, 
in which case the degree map takes values in $\mathbb{N}^2$.
As an $\mathbb{N}$-graph it is acyclic, rooted, and regular, but not strongly simple. 
However as an $\mathbb{N}^2$-graph it is rooted, strongly simple and regular. 
To see the difference in strong simplicity between the two cases consider the objects $(0,0)$ and $(1,1)$ and the possible morphisms $(0,0)\to(1,1)$. 
In both cases we can decompose the morphism $(0,0)\to(1,1)$ as either $(0,0)\to(1,0)\to(1,1)$ or $(0,0)\to(0,1)\to(1,1)$.
As an $\mathbb{N}$-graph these are both compositions of morphisms of degree $1$, the generator of $\mathbb{N}$. 
Hence for the factorization property to hold, the compositions $(0,0)\to(1,0)\to(1,1)$ and $(0,0)\to(0,1)\to(1,1)$ must correspond to two different morphisms $(0,0)\to(1,1)$ and the $\mathbb{N}$-graph is not strongly simple. 
It is, however, acyclic as there are no directed loops in the graph.
As an $\mathbb{N}^2$-graph, $(0,0)\to(1,0)\to(1,1)$ is 
a composition of morphisms first of degree $(1,0)$ then of degree $(0,1)$, 
whereas $(0,0)\to(0,1)\to(1,1)$ is a composition of morphisms first of degree $(0,1)$ then of degree $(1,0)$. 
These can therefore be identified without violating the factorization property to give a unique morphism $(0,0)\to(1,1)$. 
As an $\mathbb{N}^2$-graph it is possible to identify all directed paths from one object to another without violating the factorization property, and the $\mathbb{N}^2$-graph is strongly simple.
There is a geometric perspective on this distinction. 
Geometrically, viewing $\mathbb{N}^2$ as an $\mathbb{N}^2$-graph corresponds to thinking of $\mathbb{N}^2$ as a rank~1 cubical complex, whereas as an $\mathbb{N}$-graph we consider only the 1-skeleton as a directed graph.

More generally, any product of $k$ regular rooted trees is 
an acyclic, rooted, strongly simple and regular $\mathbb{N}^k$-graph by the construction in Section~\ref{sec:ProductsofPgraphs}.
\end{remark}

\begin{remark}
\label{characterisation of rooted, strongly-simple}
A $P$-graph $(\mathscr{L},d)$ is both rooted and strongly-simple if and only if 
there exists a unique $\alpha\in\Obj(\mathscr{L})$ such that for every $\beta\in\Obj(\mathscr{L})$ 
there exists a unique $\lambda\in\Hom(\mathscr{L})$ with $\lambda\colon \alpha\mapsto\beta$.
\end{remark}

\begin{proposition}
\label{prop:regular_P-graph}
Let $P$ be a finitely generated semigroup that embeds in a group and satisfies $P\cap P^{-1} = \triv$. Let $\Sigma$ be a generating set for $P$. 
Suppose that $(\mathscr{L},d)$ is a rooted, regular, strongly simple $P$-graph and  
denote the root of $\mathscr{L}$ by $\nu_0$. 
\begin{enumerate}
\item For each $\alpha\in \hbox{\rm Obj}(\mathscr{L})$ there is 
a unique morphism $\lambda :\nu_0 \to \alpha$.  
Define the \emph{level} of $\alpha$ to be $\ell(\alpha) = d(\lambda)$.
\label{prop:regular_P-graph1}
\item If $x,y\in d(\mathscr{L})$, then $xy\in d(\mathscr{L})$.
\label{prop:regular_P-graph2}
\item There is a subset $\Sigma'$ of $\Sigma$ such that 
$d(\mathscr{L})$ is the subsemigroup of $P$ generated by $\Sigma'$. 
If $\Sigma' = \emptyset$, then $\mathscr{L} = \{\nu_0\}$. 
Otherwise, $\mathscr{L}$ is infinite. 
\label{prop:regular_P-graph3}
\end{enumerate}
\end{proposition}
\begin{proof}
(\ref{prop:regular_P-graph1}) 
Existence of $\lambda$ follows because $\nu_0$ is the root and 
uniqueness because $\mathscr{L}$ is strongly simple.

(\ref{prop:regular_P-graph2}) 
Let $\lambda : \alpha\to\beta$ be such that $d(\lambda) = x$ and 
$\mu : \gamma\to\delta$ be such that $d(\mu) = y$. 
Then, since $\mathscr{L}$ is regular, 
there is an isomorphism $\phi : \mathscr{L}^\gamma \to \mathscr{L}^\beta$. 
In particular, $\phi(\gamma) = \beta$ and the morphism $\phi(\mu) : \beta\to \phi(\delta)$. Hence the composite morphism $\lambda\phi(\mu)$ satisfies 
$$
d(\lambda\phi(\mu)) = d(\lambda)d(\phi(\mu)) = xy.
$$ 

(\ref{prop:regular_P-graph3})  
If $\mathscr{L} = \{\nu_0\}$,  then $d(\mathscr{L}) = \{1\}$ and $\Sigma' = \emptyset$. 
Otherwise, there is an object $\alpha\ne \nu_0$ and 
$\lambda : \nu_0\to\alpha$ with $x = d(\lambda) \ne 1$. 
Then we have that $x = x_1\dots x_l$ with each $x_i\in \Sigma$ and, by the factorization property, for each $x_i$ there is $\alpha_i\in \hbox{Obj}(\mathscr{L})$ with $\ell(\alpha_i) = x_i$. 
Repeated application of (\ref{prop:regular_P-graph2}) then implies that, 
for every $y$ in the semigroup generated by $\{x_1,\dots,x_l\}$, 
there is $\beta\in \hbox{Obj}(\mathscr{L})$ with $\ell(\beta) = y$. 
Since this holds for every $x$ in $d(\mathscr{L})$, 
it follows that the latter is generated by a subset of $\Sigma$.
\end{proof}

\begin{remark}
Strong simplicity of $\mathscr{L}$ is required in order for 
the `level' of objects to be well-defined. 
For example, let $N$ be the directed graph that has 
$V(N) = \mathbb{N}$ and $E(N) = \{ (n,n+1), (n,n+2) \mid n\in \mathbb{N}\}$. 
Then $N$ is a rooted, regular acyclic $\mathbb{N}$-graph but 
all vertices except $0$ and $1$ may be reached from the root 
by paths of different lengths. 
\end{remark}

\begin{proposition}
\label{prop:regular_P-graph4}
Let $P$ be a commutative semigroup that embeds in a group 
and satisfies $P\cap P^{-1} = \triv$.
Let $(\mathscr{L},d)$ be a rooted, regular, strongly simple $P$-graph  
with root~$\nu_0$. 
Define
$$
V_x = \left\{ \alpha\in \hbox{\rm Obj}(\mathscr{L}) \mid \ell(\alpha) = x \right\}
$$
and put $\mathfrak{s}_i = |V_{x_i}|$ for each $x_i\in \Sigma$. 
Then, for each $x = \prod_{i=1}^n x_i^{m_i}$ in $P$,  
$|V_x| = \prod_{i=1}^n \mathfrak{s}_i^{m_i}$. 
Note this product is therefore independent of 
how $x$ is expressed as a product of generators from $\Sigma$. 
\end{proposition}

\begin{proof}
The set $V_{x_i}$ is not empty and $\mathfrak{s}_i\geq1$ if and only if $x_i\in \Sigma'$. 
The calculation of the cardinality of $V_x$ will be by induction on 
$m_1+\dots+m_n$ where $x = \prod_{i=1}^n x_i^{m_i}$. 
In the case when this sum is~$1$, $x = x_i$ for some $i$ and 
$|V_{x}| = \mathfrak{s}_i$ by definition. (We adopt the convention that $0^0=1$.)
Suppose that it has been shown for $x$ that 
$|V_x| = \prod_{i=1}^n \mathfrak{s}_i^{m_i}$ and consider $V_{xx_j}$. 
For each $\alpha\in V_x$, there is an 
isomorphism $\phi_\alpha : \mathscr{L} \to \mathscr{L}^\alpha$ 
(where we identify $\mathscr{L}^{\nu_0}$ with $\mathscr{L}$) and 
$\phi_\alpha$ injects $V_{x_j}$ into $V_{xx_j}$. 
Moreover, $\phi_\alpha(V_{x_j}) \cap \phi_\beta(V_{x_j})$ is 
empty if $\beta\ne \alpha$ because $\mathscr{L}$ is simple. 
Hence 
\begin{equation}
\label{eq:|V_x|}
|V_{xx_j}| \geq  |V_x|\mathfrak{s}_j = (\prod_{i=1}^n \mathfrak{s}_i^{m_i}) \mathfrak{s}_j.
\end{equation}
On the other hand, if $\beta\in V_{xx_j}$, there is 
$\lambda : \nu_0 \to \beta$ with $d(\lambda) = xx_j$. 
Then the factorization property implies that 
there are morphisms $\lambda_1$ with $d(\lambda_1) = x$ and 
$\lambda_2$ with $d(\lambda_2) = x_j$ such that $\lambda = \lambda_1\lambda_2$. 
Since $\hbox{cod}(\lambda_1) := \alpha$ belongs to $V_x$, 
it follows that $\beta \in \phi_\alpha(V_{x_j})$, and 
the inequality in (\ref{eq:|V_x|}) is an equality. 
\end{proof}

It follows from Propositions~\ref{prop:regular_P-graph} 
and~\ref{prop:regular_P-graph4} that 
a regular, rooted, strongly simple $\mathbb{N}$-graph is 
a regular, rooted tree in which every vertex has $\mathfrak{s}_1$ children.

\subsection{Products of $P$-graphs}
\label{sec:ProductsofPgraphs}

We will define the product of a family of general $P$-graphs.
We are particularly interested in analogues of products of trees.
We will therefore investigate the products of rooted, strongly-simple $P$-graphs in some detail.

\begin{definition}
\label{product of P-graphs}
Suppose $\{P_i\}_{i=1}^k$ is a family of $k$ semigroups and 
$\mathscr{L}_i$ is a $P_i$-graph with degree map $d_i$ for each $i$.
The \emph{(external) product graph}  $\mathscr{L}=\mathscr{L}_1\times\cdots\times\mathscr{L}_k$ 
is the product category consisting of 
\begin{eqnarray*}
\Obj(\mathscr{L}_1\times\cdots\times\mathscr{L}_k) 
& = & \Obj(\mathscr{L}_1)\times\cdots\times \Obj(\mathscr{L}_k) \\
\Hom(\mathscr{L}_1\times\cdots\times\mathscr{L}_k)
& = & \Hom(\mathscr{L}_1)\times\cdots\times\Hom(\mathscr{L}_k)
\end{eqnarray*}
with morphisms composed coordinatewise and 
with degree map 
$$
d\colon\mathscr{L}_1\times\cdots\times\mathscr{L}_k \to P_1\times\cdots\times P_k
$$ 
given by
$
d((\lambda_1,\ldots,\lambda_k))=(d_1(\lambda_1),\ldots, d_k(\lambda_k)).
$
\end{definition}

The first thing we need to know is that this construction yields a $P$-graph.

\begin{lemma}
\label{lem: product of P-graphs a P-graph}
Suppose $\{P_i\}$ is a family of $k$ semigroups and, for each $i$, $(\mathscr{L}_i, d_i)$ is a $P_i$-graph.
Then the product graph  $\mathscr{L}_1\times\cdots\times\mathscr{L}_k$ 
is a $P$-graph for the semigroup $P=P_1\times\cdots\times P_k$, with degree $d$ as in Definition~\ref{product of P-graphs}.
\end{lemma}

\begin{proof}
Let $\mathscr{L}=\mathscr{L}_1\times\cdots\times\mathscr{L}_k$.
To see that $\mathscr{L}$ 
is a $P$-graph,
it suffices to prove that $d$ satisfies the factorisation property. For this, suppose that 
$d((\lambda_1,\ldots,\lambda_k))=(x_1,\ldots,x_k)(y_1,\ldots,y_k)$
with each $x_i,y_i\in P_i$. 
Since each $\mathscr{L}_i$ is a $P_i$-graph, 
there exist unique $\gamma_i,\mu_i\in \Hom(\mathscr{L}_i)$ satisfying
$\lambda_i=\gamma_i\mu_i$, $d_i(\gamma_i)=x_i$ and $d_i(\mu_i)=y_i$.
Hence $(\gamma_1,\ldots,\gamma_k), (\mu_1,\ldots,\mu_k)\in 
\Hom(\mathscr{L})$ satisfy $(\lambda_1,\ldots,\lambda_k)=(\gamma_1,\ldots,\gamma_k)(\mu_1,\ldots,\mu_k)$, 
$$
d((\gamma_1,\ldots,\gamma_k))=(x_1,\ldots,x_k) \quad \text{ and } \quad
d((\mu_1,\ldots,\mu_k))=(y_1,\ldots,y_k).
$$
Suppose that $(\alpha_1,\ldots,\alpha_k), (\beta_1,\ldots,\beta_k)$ are any two morphisms in 
$\Hom(\mathscr{L})$ satisfying
$(\lambda_1,\ldots,\lambda_k)=(\alpha_1,\ldots,\alpha_k)(\beta_1,\ldots,\beta_k)$,
$d((\alpha_1,\ldots,\alpha_k))=(x_1,\ldots,x_k)$ and 
$d((\beta_1,\ldots,\beta_k))$ $= (y_1,\ldots,y_k)$. Then, for each $i$,
$\lambda_i=\alpha_i\beta_i$, $d_i(\alpha_i)=x_i$ and $d_i(\beta_i)=y_i$.
By the factorisation property in each $\mathscr{L}_i$, this implies
$\alpha_i=\gamma_i$ and $\beta_i=\mu_i$ for each $i$ 
 as required. 
 \end{proof}

We are particularly interested in identifying rooted, strongly-simple $P$-graphs as product graphs.
The following is a first step in our analysis.

\begin{lemma}
\label{lem: product of rooted strongly-simple is rooted strongly-simple}
Suppose $\{P_i\}_{i=1}^k$ is a family of semigroups and, for each $i$, 
$\mathscr{L}_i$ is a rooted, strongly-simple $P_i$-graph.
Then $\mathscr{L}_1\times\cdots\times\mathscr{L}_k$
is a rooted, strongly-simple $\left(P_1\times\cdots\times P_k\right)$-graph.
\end{lemma}

\begin{proof}
Let $\mathscr{L}=\mathscr{L}_1\times\cdots\times\mathscr{L}_k$. 
By Lemma~\ref{lem: product of P-graphs a P-graph}, $\mathscr{L}$ is a $P$-graph.
To prove that $\mathscr{L}$ is rooted and strongly simple  it suffices,
by Remark~\ref{characterisation of rooted, strongly-simple}, to prove that 
there exists a unique $\nu\in\Obj(\mathscr{L})$ such that for every $\alpha\in\Obj(\mathscr{L})$
there exists a unique morphism $\lambda\in\Hom(\mathscr{L})$ 
satisfying $\lambda\colon \nu\mapsto\alpha$.

For each $i$, let $\nu_i$ be the root of $\mathscr{L}_i$.
Let $(\nu_1,\ldots,\nu_k)\in\Obj(\mathscr{L})$ and consider an arbitrary
$(\alpha_1,\ldots,\alpha_k)\in\Obj(\mathscr{L})$.
Since each $\mathscr{L}_i$ is strongly simple and $\alpha_i\in \mathscr{L}_i$,
there exists a unique $\lambda_i\in\Hom(\mathscr{L}_i)$ 
such that $\lambda_i\colon \nu_i\mapsto\alpha_i$.
Hence $(\lambda_1,\ldots,\lambda_k)\in \Hom(\mathscr{L})$ and satisfies
$(\lambda_1,\ldots,\lambda_k) \colon (\nu_1,\ldots,\nu_k) \mapsto (\alpha_1,\ldots,\alpha_k)$.
Suppose $(\mu_1,\ldots,\mu_k)\in \Hom(\mathscr{L})$ satisfies 
$(\mu_1,\ldots,\mu_k) \colon (\nu_1,\ldots,\nu_k) \mapsto (\alpha_1,\ldots,\alpha_k)$.
Then $\mu_i\colon \nu_i\mapsto \alpha_i$ for each $i$, and 
the strong simplicity of $\mathscr{L}_i$ implies that $\mu_i=\lambda_i$ for each $i$. 
Hence $(\mu_1,\ldots,\mu_k)=(\lambda_1,\ldots,\lambda_k)$, there is a unique morphism 
$(\lambda_1,\ldots,\lambda_k)\in \Hom(\mathscr{L})$ satisfying 
$(\lambda_1,\ldots,\lambda_k) \colon (\nu_1,\ldots,\nu_k) \mapsto (\alpha_1,\ldots,\alpha_k)$, and 
$\mathscr{L}$ is a rooted, strongly-simple $P$-graph.
\end{proof}

\subsection{$P$-graphs associated with subsemigroups of flat groups}

Throughout this sect\-ion: $H$ is a finitely generated subgroup of the totally disconnected, locally compact group $G$, and $H$ is isomorphic to 
$\mathbb{Z}^k$ for some $k\in\mathbb{N}$; $U$ is a compact, open subgroup of $G$ tidy for~$H$; and $P$ is a finitely generated subsemigroup of $H$ with $P\cap P^{-1} = \triv$  that is $s$-multiplicative over~$U$.
We construct a $P$-graph in terms of $U$ and $P$ with a specific structure. 
Just as M\"oller's construction begins with the copy $\{\nu x^n \mid n\geq0\}$ of $\mathbb{N}$ in the $U$-coset space, ours begins with a copy of $P$ viewed as a $P$-graph as follows.

Construct a category  $\mathscr{P}$ from $P$ such that 
\begin{eqnarray*}
\Obj(\mathscr{P}) &=& P\\
 \text{ and } \Hom(\mathscr{P}) &=& 
 \left\{ (y_1,y_2)\in P\times P \mid \exists y'\in P \hbox{ with }y_2 = y_1y'  \right\}.
\end{eqnarray*}
Define a functor $d\colon \mathscr{P} \to P$ by setting
$d(y)$ equal to the identity of $P$ for all $y\in \Obj(\mathscr{P})$ and, 
for $(y_1,y_2)\in \Hom(\mathscr{P})$ with $y_2 = y_1y'$, setting
$d((y_1,y_2))=y'$. This is well-defined because $P$ is left cancellative.
To prove that the factorisation property is satisfied,
suppose $(y_1,y_3)\in \Hom(\mathscr{P})$ and 
$d(y_1,y_3)=\lambda_1\lambda_2$ for some $\lambda_1,\lambda_2\in P$.
Hence $y_3=y_1\lambda_1\lambda_2\in P$.
Let $y_1\lambda_1=y_2\in P$.
Then $(y_1,y_2),(y_2,y_3)\in \Hom(\mathscr{P})$, and 
$(y_1,y_2)(y_2,y_3)=(y_1,y_3)$, with 
$d((y_1,y_2))=\lambda_1$ and  $d((y_2,y_3))= \lambda_2$ as required.
Hence $\mathscr{P}$ is a $P$-graph.

The $1$-skeleton of~$\mathscr{P}$ is 
the Cayley graph for~$P$ for which the generating set is all of~$P$. 
Let $\Sigma = \{x_1, \dots, x_n\}$ be the smallest generating set for~$P$. 
This exists and is finite by Proposition~\ref{prop:gen_set}. 
Let~$X$ denote the Cayley graph of~$P$ with respect to $\Sigma$, that is, $X$ is the directed graph that has $V(X)=P$ and $E(X)= \left\{ (x,x_ix) \mid x\in P,\ x_i\in \Sigma\right\}$. 
Then~$X$ is a subgraph of the $1$-skeleton of~$\mathscr{P}$. 

For $x\in P$, the descendant  $P$-graph of $\mathscr{P}$ is the category 
$\mathscr{P}^x$ having objects 
$$
\text{desc}(x) = \left\{ y\in P \mid y = xy' \text{ for some }y'\in P\right\}
$$
and morphisms all pairs $(y_1,y_2)$ from $\mathscr{P}$ such that 
$(y_1,y_2)\in \text{desc}(x)\times \text{desc}(x)$. 
The map $y\mapsto xy$, $y\in P$,  
is an isomorphism from $\mathscr{P}$ to $\mathscr{P}^x$. 
By considering the $1$-skeleton of $\mathscr{P}^x$, 
we identify a subgraph $X^x$ of $X$ with vertex set $\text{desc}(x)$ and 
edge set $E(X)\cap \text{desc}(x)^2$.
The map $y\mapsto xy$, $y\in P$,  
also gives a graph isomorphism from $X$ to $X^x$.

The analysis of our construction requires the following technical results 
which it will be less distracting to prove separately.

\begin{lemma}
\label{lem:subgroups_line_up}
Let $P$ and $U$ be as above and suppose that $x,y\in P$.  Then 
$$
U\cap (xy)^{-1}Uxy \leq U\cap x^{-1}Ux.
$$ 
\end{lemma}
\begin{proof}
It may be assumed without loss of generality that $P$ is a maximal $s$-multiplicative semigroup over~$U$. The hypotheses then allow us to apply Proposition~\ref{prop:common_factor} and conclude that $U=U_+U_-$, where $z^{-1}U_+z\leq U_+$ and $z^{-1}U_-z\geq U_-$ for every $z\in P$. Since $U\cap z^{-1}U_-z = U_-$ and $U\cap z^{-1}U_+z = z^{-1}U_+z\leq U$, it follows that 
$$
U\cap z^{-1}Uz = U\cap (z^{-1}U_+z)( z^{-1}U_-z) = (z^{-1}U_+z) U_-
$$ 
for every $z\in P$. Hence
$$
U\cap (xy)^{-1}Uxy = y^{-1}(x^{-1}U_+x)y U_- \leq (x^{-1}U_+x) U_- = U\cap x^{-1}Ux
$$
as claimed. 
\end{proof}

\begin{lemma}
\label{lem:double cosets}
Suppose $H\cong\ZZ^k$ is a discrete subgroup of a totally disconnected, locally finite group $G$ and that~$U$ is tidy for~$H$. Suppose $P$ is a subsemigroup of $H$  that is multiplicative over~$U$. Then the double cosets $UxU$ and $UyU$ are disjoint for any distinct~$x$ and~$y$ in~$P$. 
\end{lemma}
\begin{proof}
Assume that $UxU$ and $UyU$ are not disjoint. Then, in fact, $UxU = UyU$ and there are $u,v\in U$ such that $x = uyv$. We may assume that $u\in U_-$ and $v\in U_+$, where~$U_+$ and~$U_-$ are the subgroups found in Proposition~\ref{prop:common_factor} satisfying that $U = U_+U_-$ and that $y^{-1}U_+y\leq U_+$ and $yU_-y^{-1}\leq U_-$. Then $v = y^{-1}u^{-1}x$ and so commutativity of~$H$ implies that, for every $n\geq0$,
$$
y^nvy^{-n} = y^{-1}(y^nu^{-1}y^{-n})x \in y^{-1}U_-x,
$$
which is compact. Since~$U$ is tidy for~$y$ it follows by \cite[Lemma~9]{Wi:structure} that $y^nvy^{-n}\in U$ for every~$n$ and hence that $xy^{-1} = uyvy^{-1}\in U$. Since~$H$ is discrete and torsion-free, while~$U$ is a compact subgroup of~$G$, it follows that $xy^{-1} = \ident$, that is,~$x$ and~$y$ are not distinct. 
\end{proof}

We are now in a position to state and prove our main result.

\begin{theorem}
\label{thm:P-graph G}
Suppose $H\cong\ZZ^k$ is a discrete finitely generated subgroup of a totally disconnected, locally compact group $G$ and that~$U$ is tidy for~$H$. Let $P$ be a subsemigroup of $H$ with a minimal finite generating set $\Sigma = \{x_1, \dots, x_n\}$ and suppose that $P$ is multiplicative over $U$.
Let $\Omega\equiv U\backslash G$, the set of right $U$-cosets in $G$, write $\nu$ for the element $U\in\Omega$ and $\nu x U = \left\{ \nu x u \in \Omega \mid u\in U\right\}$.
Define a category $\mathscr{G}$ by 
\begin{eqnarray*}
\Obj(\mathscr{G}) &=&  \bigcup_{x\in P} \nu x U\\
\text{ and } \quad
\Hom(\mathscr{G}) &=& \bigcup \left\{ (\nu x,\nu xy')U \mid x, y'\in P\right\}
\end{eqnarray*}
with composition $(\nu xu,\nu yu)(\nu sv,\nu zv)$ 
being defined when $\nu yu=\nu sv$, in which case
\[
(\nu xu,\nu yu)(\nu yu,\nu zu) = (\nu xu,\nu zu) 
\quad \text{ for } \quad x,y=xy',z=yz'\in P, u\in U. 
\]
Define $d\colon \mathscr{G} \to P$ by $d(x)=1$ and 
$d(\nu xu,\nu xy'u)=y'$ for $x,y'\in P$ and $u\in U$.

Then $\mathscr{G}$ is a regular, rooted, strongly-simple $P$-graph 
whose $1$-skeleton is a graph $\Gamma_P$ with 
\begin{eqnarray*}
V(\Gamma_P) &=&  \bigcup_{x\in P} \nu x U\\
\hbox{ and }\quad 
E(\Gamma_P) &=& 
\bigcup \left\{ (\nu x,\nu xx_i)U \mid x\in P,\,x_i\in \Sigma\right\}.
\end{eqnarray*}
\end{theorem}
\begin{proof}
Lemma~\ref{lem:double cosets} implies that the union $\bigcup_{x\in P} \nu x U$ defining $\Obj(\mathscr{G})$ is disjoint. Hence $\nu xu$ and~$\nu xy'u$ determine~$x$ and~$xy'$ uniquely and the degree map is well-defined. 

We begin by making some general observations about $\mathscr{G}$ 
that will help in establishing that it is a regular, rooted, strongly-simple $P$-graph.

Note that $U$ acts on $\mathscr{G}$ (and $\Gamma_P$) by multiplication on the right. 
For each $x\in P$, the $U$-orbit $\nu x U$ is a finite subset of $\Omega$
by Lemma~\ref{lem:double_coset}. 
Indeed, the stabiliser of $\nu x$ is $U\cap x^{-1}Ux$. 
Hence the map $u \mapsto \nu xu : U \to \nu xU$ induces a bijection 
\begin{equation}
\label{eq:objects_as_cosets}
\mathsf{b}_x : (U\cap x^{-1}Ux)\backslash U \to \nu x U
\end{equation}
and the cardinality of $\nu x U$ is $|U : U\cap x^{-1}Ux| = s(x)$
by  Corollary~\ref{for:cosets and scale}.
Put $\mathfrak{s}_i = s(x_i)$ for $x_i$ in the generating set $\Sigma$. 
Since $P$ is commutative, each $x\in P$ is a product, 
$x = \prod_i x_i^{m_i}$ for $m_1, \dots, m_n\in\mathbb{N}$, 
and multiplicativity of $P$ over $U$ implies that the cardinality of $\nu xU$ is
$$
|\nu x U| = \prod_{i=1}^n \mathfrak{s}_i^{m_i}.
$$
Even though the expression for $x$ in terms of $\{x_1,\dots, x_n\}$ need not be unique, 
the product $\prod_{i=1}^n \mathfrak{s}_i^{m_i}$ must be 
independent of which expression is used because $|\nu x U| $ is. 

It follows from Lemma~\ref{lem:subgroups_line_up} that, if $x$ and $y = xy'$ belong to $P$, 
then 
$$U\cap y^{-1}Uy \leq U\cap x^{-1}Ux$$ 
and so there is a well-defined and surjective  
\emph{truncation} map $\mathsf{trun}_{x,y} : \nu yU \to \nu xU$ defined by 
$$
\mathsf{trun}_{x,y}(\nu yu) = \nu xu.
$$
This map is not generally injective and $\mathsf{trun}_{x,y}^{-1}(\nu x u)$ is equal to the set of morphisms $\nu xu \to \nu yU$. 
The set of morphisms between $\nu x U$ and $\nu y U$ is $(\nu x, \nu y)U$, and so we have
$\mathsf{trun}_{x,y}^{-1}(\nu xu) = (\nu x, \nu y)(U\cap x^{-1}Ux)u$, which has cardinality 
$$
|U\cap x^{-1}Ux : U\cap y^{-1}Uy|  = \prod_{i=1}^n \mathfrak{s}_i^{p_i}
$$  
where the $p_i$ are such that $y' = \prod_{i=1}^n x_i^{p_i}$. 
Given $x$, $y = xy'$, and $z = yz'$ in $P$, the truncation maps satisfy 
$$
\mathsf{trun}_{x,y} \circ \mathsf{trun}_{y,z} = \mathsf{trun}_{x,z}.
$$ 

We now proceed to prove that $\mathscr{G}$ is a regular, rooted, strongly-simple $P$-graph by verifying the conditions in reverse order.

To prove that $\mathscr{G}$ is a $P$-graph we must prove the factorisation property.
Suppose $\lambda\in\mathscr{G}$ with $d(\lambda)=y' = y_1y_2$ for some $y_1,y_2\in P$. 
Thus $\lambda= (\nu xu,\nu xy'u)$ for some $x\in P$ and $u\in U$. Then 
$\lambda_2=(\nu xy_1u, \nu xy_1y_2u), \lambda_1= (\nu xu,\nu xy_1u)\in \Hom(\mathscr{G})$ 
are uniquely defined by the conditions $\nu xy_1u = \mathsf{trun}_{xy_1,xy'}(\nu xy'u)$ and 
$\nu xu = \mathsf{trun}_{x, xy_1}(\nu xy_1u)$ respectively. 
Moreover, $d(\lambda_1)=y_1$, $d(\lambda_2)=y_2$ and  
$\lambda_1\lambda_2=\lambda$ as required for the factorisation property.
Thus $\mathscr{G}$ is a $P$-graph.

We next prove that $\mathscr{G}$  is rooted and strongly simple.
First, note that $\nu$ is the root for $\mathscr{G}$ because $\mathscr{G}^\nu=\mathscr{G}$. 
Next, for $\alpha = \nu xu, \beta = \nu yw\in \Obj(\mathscr{G})$, 
there is a morphism $\lambda : \alpha\to\beta$ only if $y=xy'$ and 
$u$ and $w$ may be chosen to be equal. 
In that case $\lambda = (\nu xu, \nu yu)$ is the unique such morphism because $\alpha$ and~$\beta$ determine~$x$ and~$y$, and hence~$y'$, uniquely. 

To see that $\mathscr{G}$ is regular, we will construct, for each $\alpha = \nu xu$ in $\Obj(\mathscr{G})$, 
an isomorphism $\phi_\alpha : \mathscr{G} \to \mathscr{G}^\alpha$. 
For this, consider $\nu yw\in \mathscr{G}$ and recall from Proposition~\ref{prop:common_factor} that 
$\nu=U = U_-U_+$ with $yU_-y^{-1} \leq U_- \leq U$ for every $y\in P$. 
Hence for each $y\in P$
$$
\nu yU = U yU_-U_+ = U(yU_-y^{-1})yU_+ = Uy U_+ =\nu yU_+.
$$ 
Thus the coset representative $w$ in $\nu yw$ may be chosen from $U_+$, as may $u$ in~$\nu xu = \alpha$. 

Next consider $(\nu yw)x = \nu yx (x^{-1}wx) = \nu xy (x^{-1}wx)$,
which belongs to $\nu xyU$ because $x^{-1}U_+x\leq U_+$. 
Moreover, it is an object in the descendant $P$-graph $\mathscr{G}^{\nu x}$ because 
$(\nu x, \nu xy).x^{-1}wx = (\nu x, \nu xy(x^{-1}wx))$ is a morphism. 
We claim that the right $U$-coset $(\nu yw)x = \nu xy (x^{-1}wx)$ 
depends only on $\nu yw$ and not on $w$. This is because the $\nu y$-stabiliser  $(U_+)_{\nu y} = y^{-1}U_+y$, 
whence $x^{-1}(U_+)_{\nu y}x = (xy)^{-1}U_+xy = (U_+)_{\nu xy}$. 
Hence
 $$
\phi_1 :  \nu yw \mapsto \nu xy (x^{-1}wx) 
 $$
 is a well-defined map $\mathscr{G} \to \mathscr{G}^{\nu x}$. 
The same considerations show that $\phi_1$ is injective. 
Surjectivity of $\phi_1$ follows from the fact that, if $\nu xyw\in \mathscr{G}^{\nu x}$, 
then $w\in (U_+)_{\nu x} = x^{-1}U_+x$. 
That $\phi_1$ is a homomorphism of $P$-graphs may be seen by noting that, 
if $\lambda = (\nu yw, \nu yzw)\in \hbox{\rm Hom}(\mathscr{G})$, then 
 $$
 (\phi(\nu yw), \phi(\nu yzw)) = (\nu xy (x^{-1}wx), \nu xyz (x^{-1}wx))
 $$ 
 belongs to $\hbox{\rm Hom}(\mathscr{G}^{\nu x})$. 
 Therefore $\phi_1 : \mathscr{G} \to \mathscr{G}^{\nu x}$ is an isomorphism of $P$-graphs. 
 The construction of the objects and morphisms in $\mathscr{G}$ implies that 
 right multiplication by elements of $U$ is an isomorphism. Hence 
 $$
\phi_2 : \nu xy w \to \nu xywu
$$
is an isomorphism from $\mathscr{G}^{\nu x}$ to $\mathscr{G}^{\nu xu}$. Therefore 
$$
\phi_\alpha = \phi_2\circ\phi_1 : \mathscr{G} \to \mathscr{G}^{\alpha}
$$
is the desired isomorphism and $ \mathscr{G}$ is regular.

The $1$-skeleton of $\mathscr{G}$ is the directed graph  $\Gamma_P$ with
\begin{eqnarray*}
V(\Gamma_P) &=&  \bigcup_{x\in P} \nu x U\\
\hbox{ and }\quad E(\Gamma_P) &=& 
\bigcup \left\{ (\nu x,\nu xx_i)U \mid x\in P,\,x_i\in \Sigma\right\}.
\end{eqnarray*}
Indeed, the morphisms of $\mathscr{G}$ can be 
constructed by repeatedly composing edges of $\Gamma_P$, so that 
$\Hom(\mathscr{G}) = \bigcup_{\ell\geq0} E^\ell(\Gamma_P)$
where $E^\ell(\Gamma_P)$ denotes the paths of length $\ell$ in $\Gamma_P$.
\end{proof}

\begin{remark}
Note that since the $U$-cosets $\nu x$, ${x\in P}$ are disjoint we can identify $\mathscr{P}$ with the full sub-$P$-graph of $\mathscr{G}$ whose objects are $\bigcup_{x\in P} \nu x$. In this way we identify $\mathscr{P} \subset \mathscr{G}$ and $X\subset \Gamma_P$. Moreover, $U$ acts on $\mathscr{G}$ and $\Gamma_P$ by right multiplication and 
$$
\mathscr{G} = \mathscr{P}.U 
\quad\text{ and }\quad 
\Gamma_P = X.U,
$$
so that $\mathscr{G}$ and $\Gamma_P$ may be viewed as being 
composed of copies of $\mathscr{P}$ and $X$ respectively.
Moreover, $ \mathscr{P}$ and $X$ can be viewed as homomorphic images of 
$\mathscr{G}$ and $\Gamma_P$ respectively.
\end{remark}

Since the scale of $x\in P$ is independent of the subgroup~$U$ tidy for~$H$, it follows that the number of right $U$-cosets in each double $U$-coset $UxU$ ($x\in H$) is independent of $U$, and also that, if $P$ is multiplicative over one subgroup tidy for~$H$, then it is multiplicative over all. We may therefore observe that:

\begin{corollary}
The $\Gamma_P$-graph $\mathscr{G}$ defined in Theorem~\ref{thm:P-graph G} is independent of the tidy subgroup~$U$ used in its construction.
\end{corollary}

\section{Examples}
\label{sec:examples}

In this section we explore some examples arising from 
 totally disconnected, locally compact groups. 
We first give explicit examples including 
one case in which $P$ is not isomorphic to $\NN^k$ and 
one in which we obtain an $\NN^k$-graph that is not a product of trees.
Then we give a sufficient criterion for the $P$-graph to be a product of trees.

\subsection{Explicit examples}

The groups in this subsection all have the form $G = \mathbb{Q}_p^k\rtimes \mathbb{Z}^l$ for some positive integers $k,l$ but with different actions of $\mathbb{Z}^l$ on $\mathbb{Q}_p^k$ in each case. The group $H = \{0\}\rtimes \mathbb{Z}^l$ is abelian and hence flat. The $s$-multiplicative subsemigroups of $H$ are described in each case and the corresponding $P$-graph described for many of them. The descriptions of the $P$-graphs are illustrated by diagrams which are the Cayley graphs of $P$ with respect to the minimal generating sets in which the vertex $x\in P$ is labelled by $d^{-1}(x)$, the inverse image of $x$ under the degree map. 

\begin{example}\label{ex:5.1}
Let $G_1 = \mathbb{Q}_p^k \rtimes \mathbb{Z}^k$, 
where the action of $\mathbb{Z}^k$ on $\mathbb{Q}_p^k$ is 
$$
(n_1, \dots, n_k).(y_1,\dots, y_k)  = 
(p^{-n_1}y_1,\dots, p^{-n_k}y_k), \quad (n_j\in \mathbb{Z}, \ y_j\in \mathbb{Q}_p).
$$
Then $H = \mathbb{Z}^k$ is a flat subgroup of $G$ and $U = \mathbb{Z}_p^k$ is tidy for $H$. 
For $x = (n_1,\dots, n_k)\in H$ put $m(x) = \sum_{n_j\geq0} n_j$. 
Then the scale of $x$ is
\begin{equation}
\label{eq:scale_of_x}
s(x) = p^{m(x)}.
\end{equation}

It follows from (\ref{eq:scale_of_x}) that the subsemigroups of $H$ 
that are multiplicative over $U$ correspond to the subsets of $\{1,\dots. k\}$. 
For each such subset, $J$ say, put
$$
P_J = \left\{ (n_1, \dots, n_k)\in \mathbb{Z}^k \mid 
n_j\geq0\hbox{ if }j\in J, \ n_j\leq 0\hbox{ if }j\not\in J\right\}.
$$
Then $P_J$ is multiplicative over $U$ and the minimal generating set for $P_J$ is, 
denoting by $e_j$ the standard basis vector whose only nontrivial entry is a~$1$ in the $j^{\text{th}}$ position,
$$
\Sigma_J = \{e_j \mid j\in J\} \cup \{-e_j \mid j\not\in J\}.
$$
It follows that $P_J \cong \mathbb{N}^k$ for each $J$ and 
$$
U_+ = \left\{ (u_j)\in \mathbb{Z}_p^k \mid u_j = 0\hbox{ if }j\not\in J\right\}\text{ and }\ U_- = \left\{ (u_j)\in \mathbb{Z}_p^k \mid u_j = 0\hbox{ if }j\in J\right\}.
$$ 
For each $x=(m_1,\ldots,m_k)\in P_J$, the stabiliser of $\nu x$ in $U$ is~$(U_+)_{\nu x}(U_-)_{\nu x}$ where
$$
(U_+)_{\nu x} = \left\{ (p^{m_j}u_j)\in \mathbb{Z}_p^k \mid u_j\in \mathbb{Z}_p,\ u_j = 0\hbox{ if }j\not\in J\right\}\text{ and }\ (U_-)_{\nu x}= U_-,
$$ 
and we have that $\nu xU = U\backslash U\bigl(\prod_{j=1}^kp^{-m_j}\mathbb{Z}_p\bigr)x$. There is a truncation map $\mathsf{trun}_{x,y}$ if $y' = yx^{-1}$ belongs to $P_J$, that is, if $x = (m_1,\dots, m_k)$ and $y = (n_1,\dots ,n_k)$ satisfy $0\leq m_j\leq n_j$ for $j\in J$ and $0\geq m_j\geq n_j$ for $j\not\in J$. Then, for $(b_j)_{j=1}^k + U\in U\backslash U\bigl(\prod_{j=1}^kp^{-n_j}\mathbb{Z}_p\bigr)$,  
$$
\mathsf{trun}_{x,y} : 
\left((b_j)_{j=1}^k + U\right)y \mapsto  \left((p^{n_j-m_j}b_j)_{j=1}^k + U\right)x.
$$

For each generator $e_j$ (so that~$j\in J$),
$$
\ker(\mathsf{trun}_{x,x+e_j}) = \left\{(b_i)_{i=1}^j+U \mid b_i \in p^{-1}\mathbb{Z}_p\text{ if }i=j\text{ and }b_i \in \mathbb{Z}_p\text{ otherwise}\right\},
$$  
which has order~$p$, while for each generator $-e_j$ (so that~$j\not\in J$), 
$$
\ker(\mathsf{trun}_{x,x-e_j}) = \left\{(0)_{i=1}^j+U\right\},
$$  
which has order~$1$. 
It then follows from the definition in Theorem~\ref{thm:P-graph G}
that the graph~$\Gamma_{P_J}$ is isomorphic to the product of rooted trees $\prod_{j=1}^k \mathcal{T}_j$, where $\mathcal{T}_j$ is the regular tree with every vertex having $p$ children if $j\in J$ and the tree with every vertex having one child (isomorphic to $\mathbb{N}$) if $j\not\in J$. 
In particular, if $J = \{1,\dots, k\}$, then $\Gamma_{P_J}$ is $\mathcal{T}^k$, 
where $\mathcal{T}$ is the tree where every vertex has $p$ children. 
If $J=\emptyset$, then $\nu xU$ has one vertex for every $x$ and 
$\Gamma_{P_J}$ is $\mathbb{N}^k$, that is, isomorphic to $P_J$. 

When $k=1$ and $J = \{1\}$, $P_J = \mathbb{N}$ and 
the graph $\Gamma_{P_J}$ is a  rooted tree 
in which every out-valency is~$p$. 
There are thus $p^j$ vertices with degree~$j$, corresponding to the cosets $p^{-j} \mathbb{Z}_p/\mathbb{Z}_p \cong C_{p^j}$. 
We have illustrated this schematically in Figure~\ref{fig:example_1a} shows the graph where the labels $C_{p^j}$ should be interpreted as a set of vertices labelled by $C_{p^j}$ with each connected to precisely one ancestor and $p$ descendents.
\begin{figure}[htbp]
\begin{center}
\begin{picture}(280,60)(0,-20)
\put(0,0){$\{0\}$}
\put(25,3){\vector(1,0){30}}
\put(65,0){$C_p$}
\put(85,3){\vector(1,0){30}}
\put(130,0){$C_{p^2} $}
\put(160,3){\vector(1,0){30}}
\put(200,0){$C_{p^3} $}
\put(220,3){\vector(1,0){30}}
\end{picture}
\caption{$\Gamma_{P_J}$ for $G = \mathbb{Q}_p\rtimes \mathbb{Z}$ and $J=\{1\}$.}
\label{fig:example_1a}
\end{center}
\end{figure}

Figure~\ref{fig:example_1b} shows the graph $\Gamma_{P_J}$ with $k=2$ and $J = \{1,2\}$, so that $P_J = \mathbb{N}^2$. This is the graph product of two rooted trees both having out valencies equal to~$p$. The vertices with degree $(n_1,n_2)$ correspond to the cosets $(p^{-n_1}\mathbb{Z}_p)/\mathbb{Z}_p \times (p^{-n_2}\mathbb{Z}_p)/\mathbb{Z}_p \cong C_{p^{n_1}}\times C_{p^{n_2}}$. When~$k=2$, there are three other multiplicative semigroups for the cases when $J = \emptyset$, $\{1\}$ and $\{2\}$. When $J = \{2\}$, for example, the degree map projects $\Gamma_{P_J}$ onto $\mathbb{N}^2$, as in Figure~\ref{fig:example_1b}, but the vertices with degree $(n_1,n_2)$ correspond to $d^{-1}(n_1,n_2) = \{0\}\times C_{p^{n_2}}$. 

The similarity between Example~\ref{ex:Cartdec} and this multiplicative semigroup and its $P$-graph may be seen by writing $G$ as $(\mathbb{Q}_p\rtimes \mathbb{Z})^k$. It is also illuminating to form direct products of different groups. Taking $G = \left(\mathbb{Q}_{p_1}\rtimes \mathbb{Z}\right)\times \left(\mathbb{Q}_{p_2}\rtimes \mathbb{Z}\right)$ with $p_1$ and $p_2$ distinct primes and $J = \{1,2\}$ yields a $P$-graph similar to that shown in Figure~\ref{fig:example_1b} but with $d^{-1}(n_1,n_2) \cong C_{p_1^{n_1}}\times C_{p_2^{n_2}}$.

\begin{figure}[htbp]
\begin{center}
\begin{picture}(280,140)(-20,0)
\put(0,0){$\{0\}$}
\put(40,3){\vector(1,0){30}}
\put(80,0){$C_p \times \{0\}$}
\put(130,3){\vector(1,0){30}}
\put(170,0){$C_{p^2} \times \{0\}$}
\put(230,3){\vector(1,0){30}}
\put(8,15){\vector(0,1){25}}
\put(100,15){\vector(0,1){25}}
\put(194,15){\vector(0,1){25}}
\put(-15,50){$\{0\}\times C_p$}
\put(40,53){\vector(1,0){30}}
\put(80,50){$C_p\times C_p$}
\put(130,53){\vector(1,0){30}}
\put(170,50){$C_{p^2} \times C_{p}$}
\put(230,53){\vector(1,0){30}}
\put(8,65){\vector(0,1){25}}
\put(100,65){\vector(0,1){25}}
\put(194,65){\vector(0,1){25}}
\put(-15,100){$\{0\}\times C_{p^2}$}
\put(40,103){\vector(1,0){30}}
\put(80,100){$C_p \times C_{p^2}$}
\put(130,103){\vector(1,0){30}}
\put(170,100){$C_{p^2}\times C_{p^2}$}
\put(230,103){\vector(1,0){30}}
\put(8,115){\vector(0,1){25}}
\put(100,115){\vector(0,1){25}}
\put(194,115){\vector(0,1){25}}
\end{picture}
\caption{$\Gamma_{P_J}$ for $G = \mathbb{Q}_p^2 \rtimes \mathbb{Z}^2$ and $J = \{1,2\}$.}
\label{fig:example_1b}
\end{center}
\end{figure}

\end{example}

The $P$-graph of a multiplicative semigroup $P$ is 
not always a product of trees even if $P$ is isomorphic to $\mathbb{N}^k$, as seen in the next example. Unlike the example just discussed, it is necessary in the next two examples that the open normal subgroup be a product of copies of $\mathbb{Q}_p$ for the same~$p$. The coprimality condition implying that $\Gamma_P$ is a product of trees given in Proposition~\ref{prop:coprime_scales} shows this necessity.  
\begin{example} 
\label{ex:multi_ex}
Let $G_2 = \mathbb{Q}_p^3 \rtimes \mathbb{Z}^2$, 
where the action of $\mathbb{Z}^2$ on  $\mathbb{Q}_p^3$ is defined by 
extending the following actions of the standard basis vectors for~$\mathbb{Z}^2$:
$$
(a,b,c)^{e_1} = (p^{-1}a,p^{-1}b,c)\hbox{ and }(a,b,c)^{e_2} = (a,p^{-1}b,p^{-1}c).
$$
Then $H = \{0\} \rtimes \mathbb{Z}^2$ 
is flat and 
$U := \mathbb{Z}_p^3\rtimes \{0\}$ is tidy for $H$. 
The factoring of $U$ described in Theorem~\ref{thm:flat group decomp} is $U=U_1U_2U_3$, where $U_j\leq \mathbb{Z}_p^3$ is non-zero in the $j^{\text{th}}$ factor and zero in the others. The action of $H$ is such that the subsets $\{1,3\}$ and $\{2\}$ cannot occur as a set~$J_P^+$ in Definition~\ref{JP+ and JP-}. The remaining subsets of $\{1,2,3\}$ do occur and the subsemigroups of $H$ multiplicative over $U$ are set out in Table~\ref{table:Example_2}. The semigroup~$P$ is described in terms of all $n_1,n_2$ such that $e_1^{n_1}e_2^{n_2}\in P$ and by their generating sets $\Sigma$. The value of the scale on $P$ is also given. 

\begin{table}[htp]
\caption{Multiplicative subsemigroups of $H =  \{0\} \rtimes \mathbb{Z}^2\leq \mathbb{Q}_p^3 \rtimes \mathbb{Z}^2 = G_2$}
\begin{center}
\begin{tabular}{|c|c|c|c|}
\hline
$J_P^+$ & $e_1^{n_1}e_2^{n_2}\in P$ & $\Sigma$ & $s(e_1^{n_1}e_2^{n_2})$ \\
\hline\hline
${\{1,2,3\}}$ & $n_1, n_2 \geq 0$ &$e_1$, $e_2$ & $p^{2n_1+2n_2}$\\
\hline
${\{1,2\}}$ & $n_1\geq n_2$, $n_2\leq0$ & $e_1$, $e_1-e_2$ & $p^{2n_1 + n_2}$\\
\hline
${\{2,3\}}$ & $n_1\leq0$, $n_2\geq -n_1$ & $-e_1+e_2$, $e_2$ & $p^{n_1 + 2n_2}$\\
\hline
${\{1\}}$ & $0\leq n_1 \leq -n_2$, $n_2\leq0$ & $e_1-e_2$, $-e_2$ & $p^{n_1}$\\
\hline
${\{3\}}$ & $n_1\leq0$, $0\leq n_2\leq -n_1$ & $-e_1$, $-e_1+e_2$ & $p^{n_2}$\\
\hline
$\emptyset$ & $n_1,n_2\leq 0$ & $-e_1$, $-e_2$ & 1\\
\hline
\end{tabular}
\end{center}
\label{table:Example_2}
\end{table}%

The graph $\Gamma_{P_\emptyset}$ is the 
directed Cayley graph of $(\mathbb{N}^2,\Sigma)$, 
that is, $\mathbb{N}^2$ with its standard generators, and 
$\mathscr{G}$ is the $\mathbb{N}^2$- graph (or $2$-graph) $\mathbb{N}^2$. 

The vertices of the graph $\Gamma_{P_{\{1,2,3\}}}$  are 
$$
V(\Gamma_{P_{\{1,2,3\}}}) = \bigcup_{(n_1,n_2)\in P_{\{1,2,3\}}} V_{(n_1,n_2)},
$$ 
where
\begin{equation}
\label{Ex2_vertices}
V_{(n_1,n_2)} =  \nu (e_1^{n_1}e_2^{n_2})U = \left( (p^{-n_1}\mathbb{Z}_p/\mathbb{Z}_p)\times (p^{-n_1-n_2}\mathbb{Z}_p / \mathbb{Z}_p) \times (p^{-n_2}\mathbb{Z}_p / \mathbb{Z}_p)\right)(e_1^{n_1}e_2^{n_2})
\end{equation}
has order $p^{2n_1+2n_2}$. The truncation map $\mathsf{trun}_{(e_1^{m_1}e_2^{m_2}),(e_1^{n_1}e_2^{n_2})} : V_{(n_1,n_2)} \to V_{(m_1,m_2)}$ 
(for $m_1\leq n_1$, $m_2\leq n_2$) is 
\begin{align}\label{p-adic truncation}
&\phantom{aa}(a + \mathbb{Z}_p, b + \mathbb{Z}_p, c + \mathbb{Z}_p)(e_1^{n_1}e_2^{n_2}) \mapsto \\
&(p^{n_1-m_1}a + \mathbb{Z}_p, p^{n_1-m_1+n_2-m_2}b + \mathbb{Z}_p, p^{n_2-m_2}c + \mathbb{Z}_p)(e_1^{m_1}e_2^{m_2}).\notag
\end{align}
It will now be seen that $\mathscr{G}$ and $\Gamma_P$ are not products of trees.

It follows from~\eqref{Ex2_vertices} that 
\begin{align*}
V_{(1,0)} &= \left( (p^{-1}\mathbb{Z}_p/\mathbb{Z}_p)\times (p^{-1}\mathbb{Z}_p / \mathbb{Z}_p) \times (\mathbb{Z}_p / \mathbb{Z}_p)\right)(e_1)\\
\text{ and }V_{(1,1)} &= \left( (p^{-1}\mathbb{Z}_p/\mathbb{Z}_p)\times (p^{-2}\mathbb{Z}_p / \mathbb{Z}_p) \times (p^{-1}\mathbb{Z}_p / \mathbb{Z}_p)\right)(e_1e_2)
\end{align*} 
and from~\eqref{p-adic truncation} that the edge set between 
$V_{(1,0)}$ and $V_{(1,1)}$ is the graph of $\mathsf{trun}_{(e_1),(e_1e_2)}^{-1}$ where
$$
\mathsf{trun}_{(e_1),(e_1e_2)} : (a + \mathbb{Z}_p, b + \mathbb{Z}_p, c + \mathbb{Z}_p)(e_1e_2) \mapsto 
(a + \mathbb{Z}_p, pb + \mathbb{Z}_p, pc + \mathbb{Z}_p)(e_1).
$$
Similarly, the edge set between $V_{(0,1)}$ and $V_{(1,1)}$ is the graph of 
$\mathsf{trun}_{(e_2),(e_1e_2)}^{-1}$ where
$$
\mathsf{trun}_{(e_2),(e_1e_2)} : (a + \mathbb{Z}_p, b + \mathbb{Z}_p, c + \mathbb{Z}_p)(e_1e_2) \mapsto 
(pa + \mathbb{Z}_p, pb + \mathbb{Z}_p, c + \mathbb{Z}_p)(e_2).
$$
Hence there are edges from the vertex 
$(a_0p^{-1} + \mathbb{Z}_p, b_0p^{-1} + \mathbb{Z}_p, \mathbb{Z}_p)(e_1)$ in $V_{(1,0)}$ to the $p^2$ vertices
$$
(a_0p^{-1} + \mathbb{Z}_p, b_0p^{-2} + b_1p^{-1} + \mathbb{Z}_p, c_0p^{-1} + \mathbb{Z}_p), \quad b_1,c_0\in \{0,1,\dots, p-1\},
$$
in $V_{(1,1)}$ and there are edges from the vertex 
$(\mathbb{Z}_p, b_0p^{-1} + \mathbb{Z}_p, c_0p^{-1} + \mathbb{Z}_p)(e_2)$ in $V_{(0,1)}$ to the $p^2$ vertices
$$
(a_0p^{-1} + \mathbb{Z}_p, b_0p^{-2} + b_1p^{-1} + \mathbb{Z}_p, c_0p^{-1} + \mathbb{Z}_p), \quad a_0,b_1\in \{0,1,\dots, p-1\},
$$
in $V_{(1,1)}$. 
Therefore the vertices $(a_0p^{-1} + \mathbb{Z}_p, b_0p^{-1} + \mathbb{Z}_p, \mathbb{Z}_p)$ 
in $V_{(1,0)}$ and $(\mathbb{Z}_p, b_0'p^{-1} + \mathbb{Z}_p, c_0p^{-1} + \mathbb{Z}_p)$ in $V_{(0,1)}$ 
have $p$ common descendants in $V_{(1,1)}$ if 
$b_0 = b_0'$ and no common descendants otherwise. 
In contrast, if the graph were a product of trees, 
then each vertex in $V_{(1,0)}$ would have $p^2$ descendants and 
exactly one common descendant with each of the $p^2$ vertices in $V_{(0,1)}$.

Another interpretation of $\Gamma_{P_{\{1,2,3\}}}$ may be seen in Figure~\ref{fig:example_2}. The semigroup $P_{\{1,2,3\}}$ is isomorphic to $\mathbb{N}^2$ and the diagram is the Cayley graph of $\mathbb{N}^2$ with vertices labelled by $d^{-1}(n_1,n_2) = C_{p^{n_1}}\times C_{p^{n_1+n_2}}\times C_{p^{n_2}}$, where $C_{p^m}$ denotes $p^{-m}\mathbb{Z}/\mathbb{Z}$. The 
truncation maps are the natural coordinatewise projections, 
{\it e.g.\/}, $\mathsf{trun}_{(e_1),(e_1e_2)} : C_p\times C_{p^2} \times C_p \to \{0\}\times C_p\times C_p$ by mapping: $C_p\to \{0\}$ in the first coordinate; $C_{p^2}\to C_p$ by quotienting by $pC_{p^2}$ in the second coordinate; and the identity map in the third coordinate. 
\begin{figure}[htbp]
\begin{center}
\begin{picture}(320,140)(-20,0)
\put(0,0){$\{0\}$}
\put(60,3){\vector(1,0){30}}
\put(100,0){$C_p\times C_p \times \{0\}$}
\put(190,3){\vector(1,0){30}}
\put(230,0){$C_{p^2}\times C_{p^2} \times \{0\}$}
\put(320,3){\vector(1,0){30}}
\put(8,15){\vector(0,1){25}}
\put(130,15){\vector(0,1){25}}
\put(265,15){\vector(0,1){25}}
\put(-30,50){$\{0\}\times C_p\times C_p$}
\put(60,53){\vector(1,0){30}}
\put(100,50){$C_p\times C_{p^2} \times C_p$}
\put(190,53){\vector(1,0){30}}
\put(230,50){$C_{p^2}\times C_{p^3} \times C_{p}$}
\put(320,53){\vector(1,0){30}}
\put(8,65){\vector(0,1){25}}
\put(130,65){\vector(0,1){25}}
\put(265,65){\vector(0,1){25}}
\put(-30,100){$\{0\}\times C_{p^2}\times C_{p^2}$}
\put(60,103){\vector(1,0){30}}
\put(100,100){$C_p\times C_{p^3} \times C_{p^2}$}
\put(190,103){\vector(1,0){30}}
\put(230,100){$C_{p^2}\times C_{p^4} \times C_{p^2}$}
\put(320,103){\vector(1,0){30}}
\put(8,115){\vector(0,1){25}}
\put(130,115){\vector(0,1){25}}
\put(265,115){\vector(0,1){25}}
\end{picture}
\caption{$\Gamma_{P_{J}}$ for $G_2$ and $J = \{1,2,3\}$. Not a product of trees.}
\label{fig:example_2}
\end{center}
\end{figure}
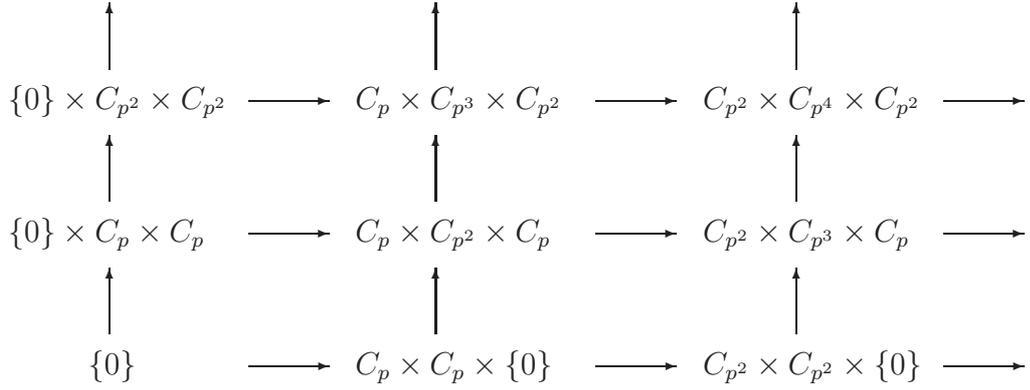
Each vertex in $C_p\times C_{p^2} \times C_p$ thus projects to vertices in $\{0\}\times C_p\times C_p$ and $ C_p\times C_p\times \{0\}$ having the same second coordinate and there are $p$ vertices in $\{0\}\times C_p\times C_p$ and $ C_p\times C_p\times \{0\}$ projecting to the same pair. On the other hand, if the graph were a product of trees, then every pair of vertices, one from $\{0\}\times C_p\times C_p$ and one from $ C_p\times C_p\times \{0\}$, would be the image under the projection of exactly one vertex in $C_p\times C_{p^2} \times C_p$. 

\end{example}

In the previous examples, the semigroup $P$ was isomorphic to $\mathbb{N}^k$. 
That that is not always the case is shown by the next example. 
\begin{example}
\label{ex:not_Nk}
Let $G_3 = \mathbb{Q}_p^2 \rtimes \mathbb{Z}^2$, 
where the action of $\mathbb{Z}^2$ on  $\mathbb{Q}_p^2$ is defined by 
extending the following actions of the standard basis vectors:
$$
(a,b)^{e_1} = (p^{-1}a,p^{-1}b)\hbox{ and }
(a,b)^{e_2} = (p^{-1}a,pb).
$$
Then $H = \{0\}\rtimes\mathbb{Z}^2)$ 
is flat and 
$U := \mathbb{Z}_p^2\rtimes \{0\}$ is tidy for~$H$. The factoring of $U$ described in Theorem~\ref{thm:flat group decomp} is $U = U_1U_2$, where $U_1$ and $U_2$ are supported on the first and second coordinates respectively. All subsets of $\{1,2\}$ may occur as $J_P^+$ in Definition~\ref{JP+ and JP-} and the corresponding semigroups are set out in Table~\ref{table:Example_3}. 
\begin{table}[htp]
\caption{ Multiplicative subsemigroups of $H = \{0\}\rtimes \mathbb{Z}^2 \leq \mathbb{Q}_p^2 \rtimes \mathbb{Z}^2 = G_3$}
\begin{center}
\begin{tabular}{|c|c|c|c|}
\hline
$J_P^+$ & $e_1^{n_1}e_2^{n_2}\in P$ & $\Sigma$ & $s(n_1,n_2)$ \\
\hline\hline
${\{1,2\}}$ & $n_1+ n_2, n_1-n_2\geq0$ &$e_1-e_2$, $e_1$, $e_1+e_2$ & $p^{2n_1}$\\
\hline
${\{1\}}$ & $n_1+ n_2 \geq 0$, $n_1-n_2\leq0$ & $-e_1+e_2$, $e_2$, $e_1+e_2$ & $p^{n_1 + n_2}$\\
\hline
${\{2\}}$ & $n_1+ n_2 \leq 0$, $n_1-n_2\geq0$ & $-e_1-e_2$, $-e_2$, $e_1-e_2$ & $p^{n_1 - n_2}$\\
\hline
$\emptyset$ & $n_1+ n_2, n_1-n_2\leq0$ & $-e_1+e_2$, $-e_1$, $-e_1-e_2$ & 1\\
\hline
\end{tabular}
\end{center}
\label{table:Example_3}
\end{table}%

For each of these semigroups, the minimal generating set $\Sigma$ has three elements and so the semigroup is not isomorphic to $\mathbb{N}^2$. The pair $(H,P_{\{1,2\}})$ is not quasi-lattice ordered in this case. For example, the pair of elements $e_1$ and $e_1-e_2$ share the common upper bounds $2e_1$ and $2e_1-e_1$  but both are minimal and so there is no least upper bound.
 
The graph $\Gamma_{P_\emptyset}$ is the Cayley graph $X(P,\Sigma)$ 
and the $P$-graph $\mathscr{G}$ is isomorphic to~$\mathscr{P}$. This is the graph shown in Figure~\ref{fig:example_3} but with single points in place of the labels given.  

We describe the graph $\Gamma_{P_{\{1,2\}}}$ in more detail. Its vertices are 
$$
V(\Gamma_{P_{\{1,2\}}}) = \bigcup_{(n_1,n_2)\in P_{\{1,2\}}} V_{(n_1,n_2)},
$$ 
where
$$
V_{(n_1,n_2)} = \nu (e_1^{n_1}e_1^{n_2})U =  \left((p^{-n_1-n_2}\mathbb{Z}_p/\mathbb{Z}_p)\times (p^{-n_1+n_2}\mathbb{Z}_p/\mathbb{Z}_p)\right)(e_1^{n_1}e_2^{n_2}),
$$
which has order $p^{2n_2}$. The truncation map $\mathsf{trun}_{(e_1^{m_1}e_2^{m_2}),(e_1^{n_1}e_2^{n_2})} : V_{(n_1,n_2)} \to V_{(m_1,m_2)}$ (for $m_1+m_2\leq n_1+n_2$ and $m_1-m_2\leq n_1 - n_2$) is 
$$
(a + \mathbb{Z}_p, b + \mathbb{Z}_p)(e_1^{n_1}e_2^{n_2}) 
\mapsto (p^{n_1-m_1+n_2-m_2}a + \mathbb{Z}_p, p^{n_1-m_1-n_2+m_2}b + \mathbb{Z}_p)(e_1^{m_1}e_2^{m_2}).
$$

Figure~\ref{fig:example_3} shows the graph $\Gamma_{P_{\{1,2\}}}$. 
The semigroup $P_{\{1,2\}}$ consists of all $(n_1,n_2)\in \mathbb{Z}^2$ with $n_1+n_2$ and $n_1-n_2$ both non-negative. In Figure~\ref{fig:example_3}, the vertex of the Cayley graph of $P_{\{1,2\}}$ with respect to the generating set $\left\{(1,1),(1,0),(1,-1)\right\}$ is labelled by $d^{-1}(n_1,n_2) = C_{p^{n_1+n_2}}\times C_{p^{n_1-n_2}}$. 

\begin{figure}[htbp]
\begin{center}
\begin{picture}(220,300)(40,-150)
\put(0,0){$\{0\}$}
\put(40,3){\vector(1,0){40}}
\put(40,13){\vector(1,1){40}}
\put(40,-10){\vector(1,-1){40}}
\put(90,0){$C_p\times C_p$}
\put(150,3){\vector(1,0){40}}
\put(150,13){\vector(1,1){40}}
\put(150,-10){\vector(1,-1){40}}
\put(200,0){$C_{p^2}\times C_{p^2} $}
\put(260,3){\vector(1,0){40}}
\put(260,13){\vector(1,1){40}}
\put(260,-10){\vector(1,-1){40}}
\put(90,60){$C_{p^2}\times \{0\}$}
\put(150,63){\vector(1,0){40}}
\put(150,73){\vector(1,1){40}}
\put(150,53){\vector(1,-1){40}}
\put(200,60){$C_{p^3}\times C_{p}$}
\put(260,63){\vector(1,0){40}}
\put(260,73){\vector(1,1){40}}
\put(260,53){\vector(1,-1){40}}
\put(200,120){$C_{p^4} \times \{0\}$}
\put(260,123){\vector(1,0){40}}
\put(260,133){\vector(1,1){40}}
\put(260,113){\vector(1,-1){40}}
\put(90,-60){$\{0\} \times C_{p^2}$}
\put(150,-57){\vector(1,0){40}}
\put(150,-47){\vector(1,1){40}}
\put(150,-67){\vector(1,-1){40}}
\put(200,-60){$C_{p}\times C_{p^3}$}
\put(260,-57){\vector(1,0){40}}
\put(260,-47){\vector(1,1){40}}
\put(260,-67){\vector(1,-1){40}}
\put(200,-120){$\{0\} \times C_{p^4}$}
\put(260,-117){\vector(1,0){40}}
\put(260,-107){\vector(1,1){40}}
\put(260,-127){\vector(1,-1){40}}
\end{picture}
\caption{$\Gamma_{P_J}$ for $G_3$ and $J = \{1,2\}$. $P_J$ is not free.}
\label{fig:example_3}
\end{center}
\end{figure}
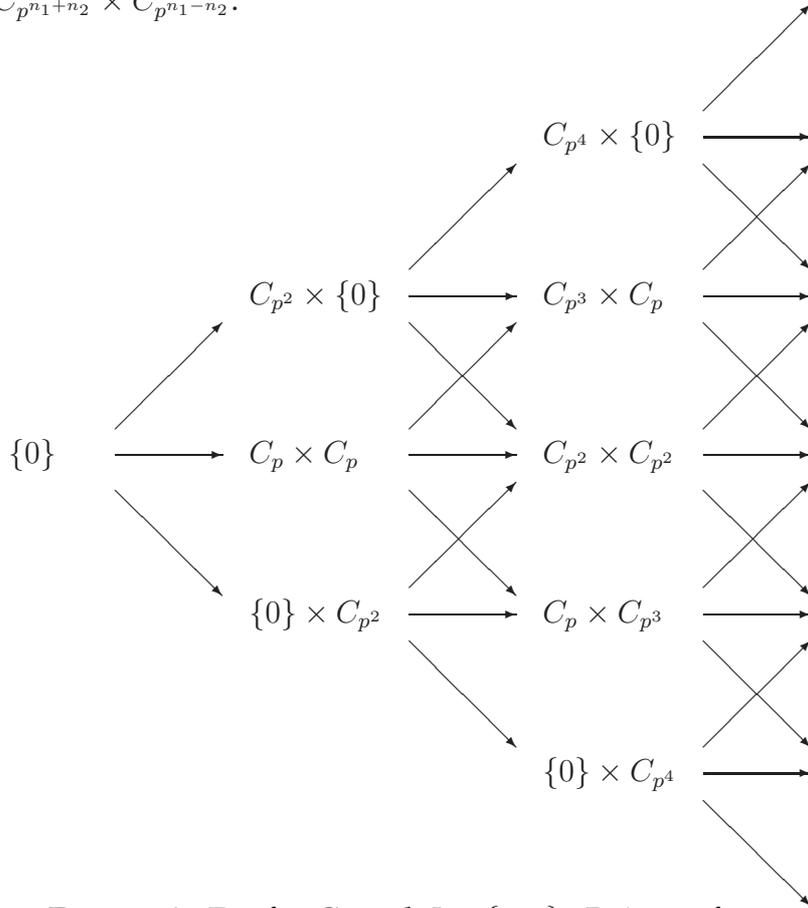

Although $\Gamma_{P_{\{1,2\}}}$ cannot possibly be a product of trees, 
it is \emph{virtually} a product of trees in the sense that 
there is an index~$2$ subsemigroup, $$
Q = \left\{ (n_1,n_2)\in P_{\{1,2\}} \mid n_1+n_2\in 2\mathbb{Z}\right\},
$$
such that $\Gamma_Q$ is a product, $\mathcal{T}_p\times \mathcal{T}_p$, 
of rooted trees where each vertex has $p$ children. 

\end{example}

\subsection{Automorphisms with coprime scales and products of trees}
\label{sec:tree_prod}

In Examples~\ref{ex:multi_ex} and~\ref{ex:not_Nk}, in which $\mathscr{G}$ and $\Gamma_P$ are not products of trees, the relative scale on each of the factors $U_j$ of the tidy subgroup $U$ is always a power of $p$. It is seen next that it is a necessary feature of any such examples that the relative scales not be coprime. 
\begin{proposition}
\label{prop:coprime_scales}
Suppose that $x,y\in G$ commute and that $s(x)s(x^{-1})$ and $s(y)s(y^{-1})$ are both not equal to~$1$ and are coprime. 
Let $U$ be tidy for the flat group $\langle x,y\rangle$. Then the tidy factoring of $U$ in Theorem~\ref{thm:flat group decomp} is
\begin{equation}
\label{eq:coprime_product}
U = U_{x+}U_{y+}U_{x-}U_{y-},
\end{equation}
where:\quad $xU_{x+}x^{-1} \geq U_{x+},\  xU_{x-}x^{-1} \leq U_{x-}\text{ and } y\text{ normalises }U_{x+}\text{ and }U_{x-}$;\\
and \qquad $yU_{y+}y^{-1} \geq U_{y+}\ yU_{y-}y^{-1} \leq U_{y-} \text{ and } x\text{ normalises } U_{y+} \text{ and } U_{y-}$.\\
Some of the factors in~\eqref{eq:coprime_product} may be trivial but at least one of $U_{x\pm}$ and one of $U_{y\pm}$ is not trivial.

When all four factors are non-trivial, there are four subsemigroups of $\langle x,y\rangle$ multiplicative over $U$ with expanding sets $J_P^+$, as in Definition~\ref{JP+ and JP-}, equal to $\{x+,y+\}$, $\{x+,y-\}$, $\{x-,y+\}$ and $\{x-,y-\}$. For each of these subsemigroups~$P$, the $P$-graph defined in Theorem~\ref{thm:P-graph G} is a product of rooted regular trees. 
\end{proposition}
\begin{proof}
The existence of the subgroup $U$ tidy for $\langle x,y\rangle$ is 
guaranteed by~\cite[Theorem~5.5]{Wi:SimulTriang}. 
By~Theorem~\ref{thm:flat group decomp}, {\it i.e.\/} \cite[Theorem~6.8]{Wi:SimulTriang}, 
there is a non-negative integer~$q$ and subgroups $U_j$, $j\in \{0,1,\dots, q\}$, of $U$ such that 
\begin{equation}
\label{eq:factorU}
U = U_0U_1\dots U_q
\end{equation}
and, for every $z$ in $\langle x,y\rangle$ and $j\in\{0,1,\dots,q\}$, 
$zU_jz^{-1}$ either contains $U_j$ or is contained in $U_j$. 
Moreover, by Lemma~\ref{details of decomp},~\cite[Theorem~6.12]{Wi:SimulTriang}, the scale of $z$ is
$$
s(z) = \prod\left\{ |zU_jz^{-1}: U_j| \mid zU_jz^{-1}\geq U_j\right\}.
$$
Since $s(x)s(x^{-1})$ and $s(y)s(y^{-1})$ are coprime, the sets
\begin{eqnarray*}
X^+ &=& \left\{ j\in\{1,\dots,q\} \mid xU_jx^{-1} > U_j \right\},  \\ Y^+ &=& \left\{ j\in\{1,\dots,q\} \mid yU_jy^{-1} > U_j \right\}, \\
X^- &=& \left\{ j\in\{1,\dots,q\} \mid xU_jx^{-1} < U_j \right\}, \\ \hbox{and }\  Y^- &=&  \left\{ j\in\{1,\dots,q\} \mid yU_jy^{-1} < U_j \right\}
\end{eqnarray*}
are pairwise disjoint. 

The element $z := xy$ satisfies $zU_jz^{-1} > U_j$ for every $j\in X^+\cup Y^+$ 
and $zU_jz^{-1} < U_j$ for every $j\in X^-\cup Y^-$. 
Hence factoring $U$ as $U = U_+U_-$, 
where $U_+ = \bigcap_{n\geq0} z^{n}U z^{-n}$ and 
$U_- = \bigcap_{n\geq0} z^{- n}U z^{n}$, 
the factors $U_j$ in (\ref{eq:factorU}) may be arranged so that 
all factors in $X^+\cup Y^+$ appear first. 
Also, by Lemma~\ref{details of decomp}~(1),  $\bigcup_{n\in\mathbb{Z}} z^nU_+z^{-n}$ is a closed, 
$\langle x,y\rangle$-stable subgroup of $G$. 
Restricting the conjugation action of $\langle x,y\rangle$ to this subgroup, $U_+$ is tidy. 
The element $w = xy^{-1}$ satisfies $wU_jw^{-1}>U_j$ 
for every $j\in X^+$ and $wU_jw^{-1}<U_j$ for every $j\in Y^+$. 
Hence, putting $U_{x+} = \bigcap_{n\geq0} w^nU_+w^{-n}$ 
and $U_{y+} = \bigcap_{n\geq0} w^{-n}U_+w^{n}$ yields $U_+ = U_{x+}U_{y+}$. 
That $U_-$ factors as $U_- = U_{x-}U_{y-}$ may be shown similarly. Every element of $\langle x,y\rangle$ either expands or shrinks every one of the factors $U_{x\pm}$ and $U_{y\pm}$. Hence no further refinement of the factoring of~$U$ occurs and Equation~\eqref{eq:factorU} is in fact just Equation~\eqref{eq:coprime_product}.
 
When all four factors in Equation~\eqref{eq:coprime_product} are not trivial, any element $x^{n_1}y^{n_2}$ with $n_1$ and $n_2$ non-negative expands the factors $U_{x+}$ and $U_{y+}$. Hence, the semigroup $P$ of all such elements is multiplicative over~$U$ and $J_P = \{x+,y+\}$. The other semigroups multiplicative over~$U$ may be described similarly. Each of these semigroups is isomorphic to $\mathbb{N}^2$ with minimal generating sets $\{x,y\}$, $\{x,y^{-1}\}$, $\{x^{-1},y\}$ and $\{x^{-1},y^{-1}\}$ respectively. We are now in essentially the same position as in Examples~\ref{ex:Cartdec} and~\ref{ex:multi_ex} and the $P$-graph for each of these semigroups is a product of trees. In the case of $P = \left\{ x^{n_1}y^{n_2} \mid n_1,n_2\geq0\right\}$ for instance, we have  $d^{-1}(n_1,n_2) = C_{s^{n_1}}\times C_{s^{n_2}}$, where $s(x) = s$ and $s(y) = t$, and $\Gamma_P$ is isomorphic to the product of regular rooted trees with out-valencies~$s$ and~$t$ respectively.
\end{proof}

\section{Comments and Questions}
\label{sec:questions}

\begin{remark}
In~\cite{struc(tdlcG-graphs+permutations)}, M\"oller works only with positive powers of $x$. 
The fact that a subgroup $U$ that is tidy for $x$ is 
also tidy for $x^{-1}$ and for $\langle x\rangle$ means that 
checking multiplicativity of positive powers of $x$ over $U$ 
is enough to characterize tidiness for the flat group~$\langle x\rangle$. 

Multiplicativity over a semigroup does not imply tidiness. 
Example~3.5 in~\cite{Wi:SimulTriang} gives a flat group generated by 
commuting automorphisms $\alpha_1$, $\alpha_2$ and 
compact, open subgroup, $U$, such that 
the semigroup generated by $\alpha_1$ and $\alpha_2$ is multiplicative over $U$ 
but $U$ is not tidy for the group $\langle\alpha_1, \alpha_2\rangle$. 
In this example, $\alpha_1$ and $\alpha_2$ do not generate a free semigroup 
because $\alpha_1^2 = \alpha_2^2$. 
It may be that if the automorphisms are sufficiently independent and 
the semigroup they generate is multiplicative over $U$ 
then $U$ is tidy for the group that they generate. 
\end{remark} 

\begin{remark}
Suppose that $H$ is simply a subgroup of $G$ (not necessarily abelian) and 
$U$ is a compact open subgroup. Suppose that $P$ is a subsemigroup of $H$. 
Can a similar construction using cosets of $U$ be made to work? 
Can multiplicativity of $P$ over $U$ be characterized in terms of this graph? 
If $H$ is flat and $P$ is multiplicative over $U$, 
then~$U$ factors as $U = U_+U_-$, 
the image of $P$ in $H/H(1)$ is a subsemigroup of $\mathbb{Z}^k$, and 
the graph (or $P$-graph) can be defined in terms of the coset spaces $U_+/(x^{-1}U_+x)$, $x\in P$. 

It may be shown that finitely generated nilpotent groups are flat and so a version for $H$ nilpotent at least is desirable.
\end{remark}

\begin{remark}
The notion of a regular, rooted, strongly simple $P$-graph 
(over the subsemigroups of $\mathbb{Z}^k$ arising here) is defined abstractly. 
They are not all products of trees, even when $P$ is a product of $k$ copies of~$\mathbb{N}$. 
Can these $P$-graphs be classified? 

By construction, these $P$-graphs are acted on by $U$. 
In the examples $U = \mathbb{Z}_p^q$ for some $q>0$. 
What are the automorphism groups of regular, rooted, strongly simple $P$-graphs? 
Are they much larger than $\mathbb{Z}_p^q$ in the higher rank cases? 

\end{remark}

\end{document}